\documentclass[11pt,reqno]{amsart}
\usepackage{hyperref,amsmath,amssymb,graphicx,amscd}
\usepackage[latin1]{inputenc}
\usepackage{mathrsfs}
\usepackage{enumerate}
\usepackage{mathrsfs}
\usepackage{dsfont}
\usepackage{amsfonts}
\newcommand{\CC}{{\mathbb C}}
\newcommand{\DD}{{\mathbb D}}

\newcommand{\RR}{{\mathbb R}}

\newcommand{\NN}{{\mathbb{N}}}

\renewcommand{\span}{{\mathrm{span}}}

 \DeclareMathOperator{\Prob}{\mathbb{P}}   
 \DeclareMathOperator{\E}{\mathbb{E}}      

 \newcommand{\dd}{{\mathrm{d}}}            
 \newcommand{\ii}{{\mathrm{i}}}
  \newcommand{\Id}{{\mathrm{Id}}}
\newcommand{\law}{\overset{\mbox{\rm \scriptsize law}}{=}}

\renewcommand{\Re}{{\mathfrak{Re}}}
\renewcommand{\Im}{{\mathfrak{Im}}}

\def \be{\begin{eqnarray*}}
\def \ee{\end{eqnarray*}}
\def \ben{\begin{eqnarray}}
\def \een{\end{eqnarray}}

\renewcommand{\det}{\mathrm{det}}

\def\sn{^{(n)}}
\def \w{{\tt sp}}
\def \u{{\tt esd}}

\newtheorem{thm}{Theorem}[section]

\newtheorem{lem}[thm]{Lemma}
\newtheorem{prop}[thm]{Proposition}
\theoremstyle{definition}
\newtheorem{defn}[thm]{Definition}
\theoremstyle{remark}
\newtheorem*{rem}{Remark}

\numberwithin{equation}{section}

\numberwithin{equation}{section}

\begin{document}
\title[Circular Jacobi Ensembles and deformed 
 coefficients]
{Circular Jacobi Ensembles and deformed Verblunsky coefficients}

\author{P. Bourgade}
 \address{ENST, 46 rue Barrault, 75634 Paris Cedex 13.
Universit\'e Paris 6, LPMA, 175,
rue du Chevaleret F-75013 Paris.}
 \email{bourgade@enst.fr}

\author{A. Nikeghbali}
 \address{Institut f\"ur Mathematik,
 Universit\"at Z\"urich, Winterthurerstrasse 190,
 CH-8057 Z\"urich,
 Switzerland}
 \email{ashkan.nikeghbali@math.uzh.ch}

\author{A. Rouault}
 \address{Universit\'e  Versailles-Saint Quentin, LMV,
 B\^atiment Fermat, 45 avenue des Etats-Unis,
78035 Versailles Cedex}
 \email{alain.rouault@math.uvsq.fr}


\begin{abstract}
Using the spectral theory of unitary operators and the theory of
orthogonal polynomials on the unit circle, we propose a simple
matrix model  for the following circular analogue of the Jacobi
ensemble:
$$c_{\delta,\beta }^{(n)} \prod_{1\leq k<l\leq n}|
e^{\ii\theta_k}-e^{\ii\theta_l}|^\beta\prod_{j=1}^{n}(1-e^{-\ii\theta_j})^{\delta}
(1-e^{\ii\theta_j})^{\overline{\delta}}\,$$ with $\Re\!\ \delta >
-1/2$. If $e$ is a cyclic vector for a unitary $n\times n$ matrix
$U$, the spectral measure of the pair $(U,e)$ is well parameterized
by its Verblunsky coefficients $(\alpha_0, \dots, \alpha_{n-1})$. We
introduce here a deformation $(\gamma_0, \dots, \gamma_{n-1})$ of
these coefficients so that the associated Hessenberg   matrix
(called GGT) can be decomposed into a product $r(\gamma_0)\cdots
r(\gamma_{n-1})$ of elementary reflections parameterized by these
coefficients. If  $\gamma_0, \dots, \gamma_{n-1}$ are independent
random variables with some remarkable distributions, then the
eigenvalues of the GGT matrix follow the circular Jacobi
distribution above.

 These deformed Verblunsky coefficients also allow to prove that, in the regime  $\delta = \delta(n)$ with $\delta(n)/n \rightarrow \beta\dd/2$,    the spectral measure and the empirical spectral distribution  weakly converge to  an
 explicit nontrivial probability measure supported by an arc of the unit
 circle. We also prove the large deviations for the empirical spectral distribution.

\end{abstract}

\maketitle

\tableofcontents

\section{Introduction}

\subsection{The circular Jacobi ensemble.}
The theory of random unitary matrices was  developed using the
existence of a natural probability uniform measure on compact Lie
groups, namely the Haar measure.  The statistical properties of the
eigenvalues as well as the characteristic polynomial of these random
matrices have played a crucial role both in physics (see
\cite{Mehta} for an historical account) and in analytic number
theory to model $L$-functions (see \cite{KeaSna} and \cite{KeaSna2}
where Keating and Snaith  predict moments of $L$-functions on the
critical line using knowledge on the moments of the characteristic
polynomial of random unitary matrices).

The circular unitary ensemble (CUE) is $U(n)$, the unitary group over $\mathbb C^n$, equipped with its Haar probability measure $\mu_{U(n)}$.
Weyl's integration formula allows one to
average any (bounded measurable) function on $U(n)$ which is conjugation-invariant
\begin{equation}\label{eqn:Weyl integration}
\int f d\mu_{U(n)}
=\frac{1}{n!}\idotsint
|\Delta(e^{\ii\theta_1},\dots,e^{\ii\theta_n})|^2 f(\hbox{diag}\!\ (e^{\ii\theta_1},
\dots,e^{\ii\theta_n}))
\frac{d\theta_1}{2\pi}\dots \frac{d\theta_n}{2\pi},
\end{equation}
where $\Delta(e^{\ii\theta_1},\dots,e^{\ii\theta_n})= \prod_{1\leq
j<k\leq n}(e^{\ii\theta_k }-e^{\ii\theta_j})$ denotes the
Vandermonde determinant.

The circular orthogonal ensemble (COE) is the subset of $U(n)$ consisting of symmetric matrices, i.e.
$U(n)/O(n) = \{ VV^T ; V \in U(n)\}$
equipped with the measure obtained by pushing forward $\mu_{U(n)}$ by the mapping $V \mapsto VV^T$.
The integration formula is similar to (\ref{eqn:Weyl integration}) but with $|\Delta(e^{\ii\theta_1},\dots,e^{\ii\theta_n})|^2$ replaced by
$|\Delta(e^{\ii\theta_1},\dots,e^{\ii\theta_n})|$ and with the
normalizing constant changed accordingly.

For the circular symplectic ensemble (CSE), which will not be recalled here, the integration formula uses $|\Delta(e^{\ii\theta_1},\dots,e^{\ii\theta_n})|^4$.

  Dyson observed that the induced
eigenvalue distributions correspond to the Gibbs distribution for
the classical Coulomb gas on the circle at three different
temperatures. More generally, $n$ identically charged particles
confined to move on the unit circle, each interacting with the
others through the usual Coulomb potential $-\log |z_i-z_j|$,
 give
rise to the Gibbs measure at temperature $1/\beta$
 (see the discussion
and references in \cite{KillipNenciu} and in \cite{ForrBook} chap. 2):
\begin{equation}
\E_n^\beta(f)= c_{0, \beta}\sn \int
f(e^{\ii\theta_1},\dots,e^{\ii\theta_n})|
\Delta(e^{\ii\theta_1},\dots,e^{\ii\theta_n})|^\beta
d\theta_1\ldots d\theta_n,
\end{equation}
where $c_{0, \beta}\sn$ is a normalizing constant chosen so that
\ben \label{eqn:beta circle} h_{0, \beta}\sn (\theta_1, \dots,
\theta_n) = c_{0, \beta}\sn
|\Delta(e^{\ii\theta_1},\dots,e^{\ii\theta_n})|^\beta \een
 is a  probability density on $(0, 2\pi)^n$
 and where $f$ is
 any symmetric function. The unitary, orthogonal and
 symplectic circular ensembles correspond to matrix models for the Coulomb
 gas  at $\beta=1,2,4$ respectively,
  but are there matrix models for
 general
  $\beta>0$ for Dyson's circular eigenvalue
 statistics?

Killip and Nenciu \cite{KillipNenciu} provided matrix models for
Dyson's circular ensemble, using the theory of orthogonal polynomials
on the unit circle. In particular, they obtained a sparse matrix model
which is five-diagonal,
 called CMV (after the names of the authors
Cantero, Moral, Vel\'asquez \cite{CMV}). In this framework, there
is not  a natural underlying measure such as the Haar measure;
the matrix ensemble is characterized by the laws of its elements.

There is an analogue of Dyson's circular ensembles on the real line:
the probability density function of the eigenvalues
$(x_1,\dots,x_n)$
for such ensembles with
temperature $1/\beta$ is proportional to
\begin{equation}\label{eqn:beta real}
|\Delta(x_1,\dots,x_n)|^\beta \prod_{j=1}^n e^{-x_j^2/2}
\end{equation}
For $\beta=1,2$ or $4$, this  corresponds to
the classical Gaussian
ensembles. Dimitriu and Edelman \cite{DumitriuEdelman}
gave a simple tridiagonal matrix model for (\ref{eqn:beta real}).
Killip and Nenciu \cite{KillipNenciu}, gave an analogue matrix model for the
Jacobi measure on
 $[-2,2]^n$,
 which is up to a normalizing constant,
\begin{equation}\label{eqn:beta segment}
|\Delta(x_1,\dots,x_n)|^\beta \prod_{j=1}^n (2-x_j)^{a}(2+x_j)^{b}
dx_1\dots dx_n,
\end{equation}
where $a,b>0$, relying
on the theory of orthogonal polynomials on the unit circle
and its links with orthogonal polynomials on the segment.
 When $a$ and
$b$ are strictly positive integers, the Jacobi measure (\ref{eqn:beta segment})
can be interpreted as
the potential $|\Delta(x_1,\dots,x_{n+a+b})|^\beta$ on $[-2,2]^{n+a+b}$
conditioned to have $a$ elements  located at $2$ and $b$ elements located at  $-2$.
Consequently, the Jacobi measure on the unit circle should be a two
parameters extension of
(\ref{eqn:beta circle}), corresponding to conditioning to have
 specific given eigenvalues. Such an analogue
was defined as the  \textit{circular Jacobi ensemble}
 in \cite{ForrBook} and
\cite{ForWI}.  If $\delta\in \mathbb R$, we recover the cJUE as in Witte and Forrester \cite{WiFor}.\\

\begin{defn}
 Throughout this paper, we note $h_{\delta,\beta
}^{(n)}$ the probability density function  on 
$(0, 2\pi)$ given
by:
\begin{equation}\label{loiHP}
h_{\delta,\beta }^{(n)}(\theta_1, \dots, \theta_n) = c_{\delta,\beta
}^{(n)}|\Delta(e^{\ii\theta_1},\dots,
e^{\ii\theta_n})|^\beta\prod_{j=1}^{n}(1-e^{-\ii\theta_j})^
{\delta}(1-e^{\ii\theta_j})^{\overline{\delta}}
\end{equation}
with $\delta\in\CC$, $\Re(\delta)>-\frac{1}{2}$.\\
\end{defn}
If $\delta\in\frac{\beta}{2}\NN$, this measure coincides with
(\ref{eqn:beta circle})
conditioned to have eigenvalues at $1$.
For $\beta=2$, such
 measures were first considered by
Hua \cite{Hua} and Pickrell \cite{Pick1}, \cite{Pick2}.
This case was also widely studied in \cite{Ner} and \cite{BO}
for its
connections with the theory of representations and in \cite{BNR}
for its analogies with  the Ewens measures on permutation groups.

One of our goals in this paper is to provide  matrix models for the
circular Jacobi  ensemble, i.e.
a distribution on $U(n)$ such that the arguments of the eigenvalues
$(e^{\ii \theta_1}, \dots, e^{\ii \theta_n})$ are distributed as in
(\ref{loiHP}).
 One can guess that additional problems may appear
because the distribution of the eigenvalues is not rotation
invariant anymore. Nevertheless, some statistical information for
the circular Jacobi ensemble can be obtained from Dyson's circular
ensemble by a sampling (or a change of probability measure) with the
help of the determinant. More precisely, let us first define the
notion of sampling.

\begin{defn}\label{hSampling}
Let $(X,\mathcal{F},\mu)$ be a probability space, and
$h:X\mapsto \RR^+$ a measurable and integrable function with $\E_\mu(h)>0$.
Then a measure $\mu'$
is said to be the $h$-sampling of $\mu$ if for all bounded measurable
functions $f$
$$
\E_{\mu'}(f)=\frac{\E_\mu(fh)}{\E_\mu(h)}.
$$
\end{defn}
If we consider a matrix model for $h_{0, \beta}\sn$, we can define\footnote{for $\Re(\delta)>-1/2$, due to an integrability constraint} a
matrix model for $h_{\delta, \beta}\sn$  by the means of a sampling, noticing that when the charges are actually the eigenvalues of a matrix $U$, then  (\ref{eqn:beta circle}) differs from (\ref{loiHP}) by a factor which is a function of  det$(\Id - U)$.
We define $\det_\delta$ for a unitary matrix $U$ as
$$
 \det_\delta(U)=\det(\Id-U)^{\overline{\delta}}\det(\Id-\overline{U})^\delta\,,
$$
and we will use this
$\det_\delta$-sampling.

 Actually we look for an effective
construction of a random matrix, for instance starting from a
reduced number of independent random variables with known
distributions. Notice that in the particular case $\beta=2$, the
density $h_{0,2}$ corresponds to eigenvalues of a matrix under the
Haar measure on $U(n)$ and the $\det_\delta$-sampling of this
measure is the Hua-Pickrell measure studied in our previous paper
\cite{BNR}.

\subsection{Orthogonal polynomials on the unit circle.}
We now wish to outline the main ingredients which are needed from the theory
of orthogonal polynomials on the unit circle to construct matrix models for
the general Dyson's circular ensemble. The reader can refer to
\cite{SimonBook1} and  \cite{SimonBook2} for more results and references;
in particular, all the results about orthogonal polynomials on the unit
circle (named hereafter OPUC) can be found in these volumes.

Let us explain why OPUC play a prominent role  in these constructions. In all this paper,  $\mathbb D$ denotes the open unit disk $\{ z \in \mathbb C : |z| <
1\}$ and $\mathbb T$ the unit circle   $\{z \in \mathbb C : |z|=1\}$.
Let $({\mathcal H}, u, e)$ be a triple where $\mathcal H$ is a Hilbert space, $u$ a unitary operator and $e$ a cyclic unit vector, i.e. $\{u^j e\}_{j=-\infty}^\infty$ is total in ${\mathcal H}$. We say that two triples  $({\mathcal H}, u, e)$ and $({\mathcal K}, v, e')$ are equivalent
if and only if there exits an isometry $k : {\mathcal H} \rightarrow {\mathcal K}$ such that
$v= kuk^{-1}$ and $e'=ke$. The spectral theorem says  that for each equivalence class, there exists a unique probability measure $\mu$ on $\mathbb T$ such that
\[\langle e, u^k e\rangle_{\mathcal H} = \int_{\mathbb T} z^k d\mu(z)\ \ , \ \ k=0,\pm1 ,\dots\,.\]
Conversely,  such a probability measure $\mu$ gives rise to a triple consisting  of the Hilbert space  $L^2(\mathbb T, \mu)$, the operator of multiplication by $z$, i.e.
$h\mapsto (z \mapsto zh(z))$ and the vector ${\bf 1}$, i.e. the constant function $1$.
When the space ${\mathcal H}$ is fixed, the probability measure $\mu$ associated with the triple $({\mathcal H}, u, e)$ is called  the {\it spectral measure of the pair} $(u, e)$.

Let us consider the finite $n$-dimensional case.
Assume that $u$ is unitary and $e$ is cyclic. It is classical that $u$ has $n$ different eigenvalues $(e^{\ii\theta_j}, j=1, \dots, n)$ with $\theta_j \in [0, 2\pi)$. In any orthonormal basis whose first vector is $e$,  say $(e_1 = e, e_2, \dots , e_n)$, $u$ is represented by a matrix $U$ and there is a unitary matrix $\Pi$ diagonalizing $U$. It is then straightforward that  the spectral measure is
\ben
\label{defmuw}
\mu = \sum_{j=1}^n \pi_j\, \delta_{e^{\ii\theta_j}}\een where
the weights are defined as $\pi_j = |\langle e_1, \Pi
e_j\rangle|^2$.
 Note
that $\pi_j>0$ because a cyclic vector cannot be orthogonal to any
eigenvector (and we also have $\sum_{j=1}^n\pi_j=1$ because $\Pi$ is
unitary). The eigenvalues $(e^{\ii\theta_j}, j=1, \dots n)$ and the
vector $(\pi_1,\ldots,\pi_n)$ can then be used as coordinates
for the probability measure $\mu$.

Keeping in mind our purpose, we see that the construction of a matrix model from a vector $(e^{\ii\theta_j}, j=1, \dots, n)$ may be achieved in two steps: first give a vector of weights $(\pi_1, \dots, \pi_n)$, then find a matricial representative of the equivalence class with a rather simple form. The key tool for the second task is the sequence of orthogonal polynomials associated with the measure $\mu$.
 In $L^2(\mathbb T, \mu)$
equipped with the natural basis
 $\{1, z, z^2, \dots, z^{n-1}\}$, the Gram-Schmidt procedure provides
 the family of monic
orthogonal polynomials $\Phi_0, \dots, \Phi_{n-1}$.
 We can still define $\Phi_n$ as the unique monic
polynomial of degree $n$ with $\parallel\!\Phi_n\! \parallel_{L^2(\mathbb T, \mu)}=0 $, namely
\ben
\label{char}
\Phi_n (z) =\prod_{j=1}^n (z - e^{\ii\theta_j})\,.
\een
The $\Phi_k$'s
($k= 0, \dots, n$)
obey the Szeg\"o recursion relation:
\begin{equation}
\label{Szego}
\Phi_{j+1}(z) = z\Phi_j(z) - \bar\alpha_j \Phi_j^*(z)
\end{equation}
where
\begin{equation}\label{phietoile}
\Phi_j^*(z) = z^j\!\ \overline{\Phi_j(\bar z^{-1})}\,.
\end{equation}
The coefficients $\alpha_j$'s ($0\leq j\leq n-1$) are called
Verblunsky coefficients and satisfy the condition $\alpha_0, \cdots
, \alpha_{n-2} \in \mathbb D$ and $\alpha_{n-1} \in \mathbb T
$.

When the measure $\mu$ has infinite support, one can define the family of orthogonal polynomials $(\Phi_n)_{n\geq0}$ associated with $\mu$ for all $n$. Then there are infinitely many Verblunsky coefficients $(\alpha_n)$ which all lie in $\DD$.

Verblunsky's Theorem (see for example \cite{SimonBook1},
\cite{SimonBook2}) states that there is a bijection
between probability measures on the unit circle and sequences of
Verblunsky coefficients.
The matrix of the multiplication by $z$ in $L^2(\mathbb T,
\mu)$,
 in
the basis of orthonormal polynomials, has received much attention.
This unitary matrix, noted
$\mathcal{G}(\alpha_0, \dots ,\alpha_{n-1})$, called GGT by B. Simon (see Chapter 4.1
in \cite{SimonBook1} for more details and for an historical
account), is in the Hessenberg form: all entries below the
subdiagonal are zero, whereas the entries above the subdiagonal are
nonzero and the subdiagonal is nonnegative (see formulae (4.1.5) and
(4.1.6)  in \cite{SimonBook1} for an explicit expression for the
entries in terms of the Verblunsky coefficients, or formula
(\ref{rien}) in Lemma \ref{lemmehessenberg} below).

For $H$ a $n\times n$ complex matrix, the subscript $H_{ij}$
stands for $\langle e_i, H(e_j)\rangle$, where $\langle x, y\rangle=
\sum_{k=1}^n \overline{x}_k y_k$.  
Killip and Nenciu
state that any unitary matrix in the Hessenberg form with nonnegative
subdiagonal is the matrix of multiplication by $z$ in $L^2(\mathbb T,
\mu)$ for some measure
$\mu$. More precisely,
from Killip-Nenciu \cite{KillipNenciu}, Lemma 3.2, we have
\begin{lem}\label{lemmehessenberg}
Let $\mu$ be a probability measure on $\mathbb T$ supported at $n$ points. Then the matrix $H$ of  $f(z)\mapsto zf(z)$ in the basis of
orthonormal polynomials of $L^2(\mathbb T, \mu)$, is in the Hessenberg form.
 More precisely \begin{equation}\label{rien}
H_{i+1,j+1} =
\begin{cases}  -\alpha_{i-1}\overline{\alpha}_j
\prod_{p=i}^{j-1}\rho_p& \text{if $i<j+1$}
\\
\rho_{j-1} &\text{if $i=j+1$} \\
0&\text{if $i>j+1$}
\end{cases},
\end{equation}with $\rho_j=\sqrt{1-|\alpha_j |^2}$ and $\alpha_{-1}=-1$, the $\alpha_k$'s being the Verblunsky coefficients associated to
$\mu$.
 Conversely, if $\alpha_0, \dots,\alpha_{n-1}$ are given in $\mathbb D^{n-1}\times \mathbb T$, and if we define the matrix $H$ by (\ref{rien}), then the spectral measure of the pair $(H, e_1)$ is  the measure $\mu$ whose Verblunsky coefficients are precisely $\alpha_0, \dots,\alpha_{n-1}$.
\end{lem}


Besides, there is a very useful decomposition of these
matrices into product of block matrices, called the AGR
decomposition by Simon (\cite{Simon5years}), after the paper
\cite{AGR}. For $0\leq k\leq n-2$, let $$
\Theta^{(k)}(\alpha)=\Id_{k}\oplus \left(
\begin{array}{cc}
\overline{\alpha}_k&\rho_k\\
\rho_k&-\alpha_k
\end{array}
\right)
\oplus
\Id_{n-k-2}.
$$
and set $\Theta^{(n-1)}(\alpha_{n-1})=\Id_{n-1}\oplus
(\overline{\alpha}_{n-1})$, with $|\alpha_{n-1}|=1$. Then the AGR
decomposition states that (\cite{Simon5years} Theorem 10.1)
\[{\mathcal G}(\alpha_0, \dots, \alpha_{n-1}) = \Theta^{(0)}
(\alpha_0)\Theta^{(1)}(\alpha_1)\ldots\Theta^{(n-1)}(\alpha_{n-1})\,.\]

Now we state a crucial result of Killip and Nenciu which enabled them to
obtain a matrix model in the Hessenberg form for Dyson's circular
ensemble. The challenge consists in randomizing $(\alpha_0, \dots, \alpha_{n-1})$ in $\mathbb D^{n-1}\times \mathbb T$ in such a way that the
angles $(\theta_1, \dots, \theta_n)$ of the spectral measure
\[\mu = \sum_{j=1}^n \pi_j \delta_{e^{\ii \theta_j}}\]
have the density $h_{0, \beta}^{(\delta)}$ (see (\ref{loiHP})).
To make the statement precise, let us introduce three definitions (for the properties of the Beta and Dirichlet distributions, see \cite{Port}).
\begin{defn}
For $a_1, a_2 > 0$, let Beta$(a_1,a_2)$ be the distribution on $[0,1]$ with density
\[\frac{\Gamma(a_1+a_2)}{\Gamma(a_1) \Gamma(a_2)} x^{a_1-1}(1-x)^{a_2-1}\]
\end{defn}
Its generalization is the following.
\begin{defn}\label{defnu}
For $n\geq2$ and $a_1, \ldots, a_n > 0$, let Dir$(a_1, \ldots, a_n)$  be the distribution on the simplex $\{(x_1, \dots, x_n)
\in [0,1]^n : \sum_{i=1}^n x_i = 1\}$ with density
\[\frac{\Gamma(a_1+\ldots + a_n)}{\Gamma(a_1)\cdots \Gamma(a_n)} \prod_{k=1}^n x_k^{a_k-1}\,.\]
If $a_1= \ldots =a_n =a$, it is called
the Dirichlet distribution of order $n\geq2$ with parameter $a>0$ and  denoted by Dir$_n (a)$.
(For $a= 1$, this is the uniform distribution).
\end{defn}
\begin{defn}
 For $s > 1$ let $\nu_s$ be the probability measure on $\mathbb D$ with density
$$\frac{s-1 }{2\pi }(1-|z|^2)^{(s-3)/2}.$$
It is the law of $\textsc{r}e^{\ii \psi}$ where $\textsc{r}$ and $\psi$ are independent, $\psi$ is uniformly distributed on $(0, 2\pi)$ and $\textsc{r}^2$ has the Beta$(1,(s-1)/2)$ distribution.
We adopt  the convention that $\nu_1$ is the
uniform distribution on the unit circle.
We denote by $\eta_{0,\beta}\sn$ the distribution on
$\mathbb D^{n-1}\times \mathbb T$ given by
\ben
\eta_{0,\beta}\sn 
 = \displaystyle\otimes_{k=0}^{n-1} \nu_{\beta(n-k-1)+1}\,.
\een
\end{defn}
\begin{prop}[Killip-Nenciu \cite{KillipNenciu}, Proposition 4.2]\label{4.2}
 The following formulae express the same measure on
 the manifold of probability distributions   on $\mathbb T$
 supported at $n$ points:
$$\frac{2^{1-n}}{n!}|\Delta(e^{\ii\theta_1},\dots,
e^{\ii\theta_n})|^\beta\prod_{j=1}^n\pi_j^{\beta/2-1 }
d\theta_1\ldots d\theta_n d\pi_1\ldots d\pi_{n-1}$$ in the
$(\theta,\pi)$ coordinates and \ben\label{eta0}\prod_{k=0}^{n-2}
(1-|\alpha_k |^2)^{(\beta/2)(n-k-1)-1} d^2\alpha_0\ldots
d^2\alpha_{n-2}\frac{\dd\phi}{2\pi }\een in terms of the Verblunsky
coefficients. 
\end{prop}
Proposition \ref{4.2}  may be restated as follows:  to pick at random a measure $\mu$ such that $(\alpha_0, \dots, \alpha_{n-1})$ is
$\eta_{0, \beta}\sn$   distributed is equivalent to pick the support  $(\theta_1, \dots, \theta_n)$ according to $h_{0, \beta}\sn$ (see (\ref{eqn:beta circle})) and to  pick the weights  $(\pi_1, \dots, \pi_n)$  independently according to   Dir$_n(\beta/2)$.

As a  consequence, if
one takes independent coefficients
$(\alpha_0,\ldots,\alpha_{n-1})$ such that $\alpha_k$
is $\nu_{\beta(n-k-1)+1}$ distributed for $0\leq k\leq n-1$,
then the GGT matrix
$\mathcal{G}(\alpha_0,\ldots,\alpha_{n-1})$ will be a
matrix model for Dyson's circular ensemble with  temperature
$1/\beta$ (see also Proposition 2.11 in \cite{ForrBook}). Actually in
\cite{KillipNenciu}, Killip and Nenciu provide a matrix model which
is much sparser 
  (five-diagonal) as shall be explained in Section
\ref{secmatrixmodel}.

Let us now define the laws on $U(n)$ which we will consider in the sequel.
\begin{defn}\label{defjaccirc}
We denote by CJ$_{0, \beta}\sn$ the probability distribution supported by  the set of $n \times n$ GGT matrices
 of the form (\ref{rien}), corresponding to the law of $\mathcal{G}(\alpha_0,\ldots,\alpha_{n-1})$ defined above.
We denote by CJ$_{ \delta,\beta}\sn$ the probability distribution on
$U(n)$ which is the det$_\delta$-sampling of CJ$_{0, \beta}\sn$.
\end{defn}

The
standard GGT approach is not sufficient to produce matrix
models  for
the circular Jacobi ensemble because, as we shall see in Section
\ref{stochastic}, under the measure CJ$_{\delta,\beta}\sn$, the
Verblunsky coefficients are not independent anymore. To overcome
this difficulty, we associate to a measure on the unit circle, or
equivalently to its  Verblunsky coefficients,   a new sequence of
coefficients $(\gamma_k)_{0\leq k\leq n-1 }$, which we shall call
deformed Verblunsky coefficients. There is a simple bijection
between the original sequence
$(\alpha_k)_{0\leq k\leq n-1 }$ and the new one
$(\gamma_k)_{0\leq k\leq n-1 }$. These coefficients satisfy among
other nice properties that $|\alpha_k |=|\gamma_k |$, and that they are independent under CJ$_{\delta,\beta}\sn$ (and
for $\delta=0$ the $\alpha_k$'s and the $\gamma_k$'s 
 have the same distribution). 
They have a
 geometric interpretation in terms of reflections: this leads
to a decomposition of the GGT matrix
$\mathcal{G}(\alpha_0,\ldots\alpha_{n-1})$ as a product of
independent elementary reflections (constructed from the
$\gamma_k$'s). The explicit expression of the densities allows an asymptotic study (as $n \rightarrow \infty$) of the $\gamma_k$'s, and consequently of the spectral measure, and finally of the empirical spectral distribution.

\subsection{Organization of the paper.}
In Section 2, after recalling basic facts about the reflections
introduced in \cite{BNR}, we define the deformed Verblunsky
coefficients $(\gamma_k)_{0\leq k\leq n-1 }$ and give some of its
basic properties. In particular we prove  that the GGT matrix
$\mathcal{G}(\alpha_0,\ldots\alpha_{n-1})$ can be decomposed into a
product of elementary complex reflections (Theorem \ref{thm:decoindptes}).

In Section 3, we derive the law of the $\gamma_k$'s under
CJ$_{\delta,\beta,}\sn$ (Thorem \ref{HPlaw}); in particular we show that they are
independent and that the actual Verblunsky coefficients are
dependent if $\delta\neq0$. We then prove an analogue of the above Proposition
\ref{4.2} on the $(\theta, \pi)$ coordinates of $\mu$ (Theorem \ref{generalisationKN4.2}).

In Section 4, we propose our matrix model (Theorem \ref{thm:new matrix model}). It is a modification of
the AGR factorization, where we transform the $\Theta_k$'s so that
they become reflections :
$$
\Xi^{(k)}(\alpha)=\Id_{k}\oplus \left(
\begin{array}{cr}
\overline{\alpha}&e^{\ii\phi}\rho\\
\rho&-e^{\ii\phi}\alpha
\end{array}
\right)
\oplus
\Id_{n-k-2},
$$
with $e^{\ii\phi}=\frac{1-\overline{\alpha}}{1-\alpha}$. Of course
the  CMV 
 representation \cite{CMV}, which is five-diagonal, is also available, but this
time the $\alpha_k$'s are not independent. Using the
following elementary fact proven in Section 2,
$$\Phi_n(1)=\det(\Id-U)=\prod_{k=0}^{n-1}(1-\gamma_k),$$
we are  able to generalize  our previous results in \cite{BHNY}
and \cite{BNR} about the decomposition of the characteristic
polynomial evaluated at $1$ as a product of independent complex variables (Proposition \ref{detJac}).

In Section 5, we study asymptotic properties of our model as $n \rightarrow \infty$, when $\delta=\beta n\dd/2$, with $\Re\!\ \dd\geq0$. We first prove that the Verblunsky
coefficients  have deterministic limits in probability. This entails  that the spectral measure
 converges weakly in
probability to the same deterministic measure  (denoted by $\mu_\dd$) which is supported by an arc of the
unit circle (Theorem \ref{cvsmalter}).
Besides, we consider the empirical spectral distribution (ESD), where the Dirac masses have the same weight $1/n$. Bounding the distances between both random measures, we prove that the ESD has the same limit (Theorem \ref{cvesdalter}). Moreover,
starting from the explicit joint distribution (\ref{loiHP}), we prove also that the  ESD satisfies  the large deviation principle at scale $(\beta/2)n^2$ whose rate function reaches its minimum at $\mu_\dd$ (Theorem \ref{LDPESDalter}).

\section{Deformed Verblunsky coefficients and reflections}
In this section, we introduce the deformed Verblunsky coefficients
and we establish some of their relevant properties, in particular
a geometric interpretation in terms of reflections. One remarkable property of the Verblunsky coefficients, as it
appears in Proposition \ref{4.2}, is that they are independent under CJ$_{0,\beta }^{(n)}$.
 As we
shall see in Section \ref{stochastic}, this does not hold anymore
under CJ$_{\delta,\beta }^{(n)}$. This motivated us to
introduce a
new set of coefficients,
$(\gamma_0,\ldots,\gamma_{n-2},\gamma_{n-1})$, called deformed
Verblunsky coefficients, which are uniquely associated with a set of
Verblunsky coefficients. In particular, $\gamma_k\in \mathbb D$ for
$0\leq k\leq n-2$, $\gamma_{n-1}\in\mathbb T$ and the map
$(\gamma_0,\ldots,\gamma_{n-1})
\mapsto(\alpha_0,\ldots,\alpha_{n-1})$ is a bijection.
Moreover, the characteristic polynomial at $1$ can be expressed
simply in terms of $(\gamma_0,\ldots,\gamma_{n-1})$.

\subsection{
Analytical properties}

Let $\mu$ be a probability measure on the unit circle supported
at $n$ points. Keeping the notations of the introduction, we let
$(\Phi_k(z))_{0\leq k\leq n}$ denote the monic orthogonal
polynomials associated with $\mu$ and $(\alpha_k)_{0\leq k\leq
n-1}$ its corresponding set of Verblunsky coefficients through
Szeg\"o's recursion formula (\ref{Szego}).
The functions
\ben\label{defb}b_k(z) = \frac{\Phi_k(z)}{\Phi_k^*(z)}\ , \ k\leq n-1\een
are known as the inverse Schur iterates (\cite{SimonBook2} p.476, after Khrushchev \cite{khru2} p.273). They are analytic in a neighborhood of $\bar{\mathbb D}$ and meromorphic in $\mathbb C$. Each $b_k$ is a finite Blashke product
\[b_k(z) = \prod_{j=1}^n \left(\frac{z-z_j}{1- \bar z_j z}\right)\]
where $z_1, \dots, z_k$ are the zeros of $\Phi_k$. Let us now explain the term "inverse Schur iterate".

The Schur function is a fundamental
 object in the study of the orthogonal polynomials on
 the unit circle.  Let us briefly recall its definition
 (see \cite{SimonBook1} or \cite{Simonopuc} for more
 details and proofs): if $\mu$ is a probability measure
 on the unit circle (supported at finitely many points or not),
 its Schur function $f: \DD\to\DD$ is defined as:
 \begin{equation}\label{schur}
f(z)=\frac{1}{z}\frac{F(z)-1}{F(z)+1} \text{ where }
F(z)=\int \frac{e^{\ii\theta}+z}{e^{\ii\theta}-z}\!\ d\mu(e^{\ii\theta}).
\end{equation}
It is a bijection between the set of probability measures on the unit circle and analytic functions mapping  $\mathbb D$ to $\bar{\mathbb D}$.
The Schur algorithm (which is described in \cite{SimonBook1} or
\cite{Simonopuc} p.438) allows to parametrize the Schur function $f$ by a sequence of so-called  Schur parameters, which are actually the Verblunsky coefficients associated to $\mu$ (Geronimus theorem).  In
particular, there are finitely many Verblunsky coefficients (or
equivalently the measure $\mu$ is supported at $n$ points) if and
only if $f$ is a finite Blaschke product. The name "inverse Schur iterate" (\cite{khru1}) for $b_k$ comes from the result (1.12) of the latter paper where $b_k$ is identified as the Schur function corresponding to the "reversed sequence" $(-\bar{\alpha}_{k-1}, \dots, -\bar\alpha_0, 1)$ (see also \cite{SimonBook2} Prop. 9.2.3).



Let us define our sequence of functions, which shall lead us the deformed coefficients.
\begin{defn} If $\mu$ is supported at $n$ points and with the notation above, define $\gamma_k(z)$ for $0\leq k\leq n-1$, as:
\begin{equation}\label{def:gammaz}
\gamma_k(z)=z-\frac{\Phi_{k+1}(z)}{\Phi_k(z)}\,.
\end{equation}
From the Szeg\"o's recursion formula (\ref{Szego}) and notation (\ref{defb}), this is equivalent to
\begin{equation}\label{def3:gammaz}
\gamma_k(z)= \frac{\bar{\alpha}_k}{b_k(z)}\,,
\end{equation}
so that $\gamma_k$ is meromorphic,
with  poles in $\mathbb D$
and  zeros lying outside  $\overline{\mathbb D}$.
\end{defn}

The next proposition shows how the functions $\gamma_k(z)$
can be defined recursively with the help of the
coefficients $\alpha_k$. As a consequence, we shall see
that the $\gamma_k(z)$ are 
 very closely related to a
fundamental object in the theory of random matrices: the
characteristic polynomial.

\begin{prop}\label{prop:gammarecursive}
For any $z\in\CC$, $\gamma_0(z)= \bar\alpha_0$ and the
following decomposition for  $\Phi_k(z)$ holds: \ben \Phi_k (z) =
\prod_{j=0}^{k-1} (z - \gamma_j(z))\ , \ \ k=1, \dots, n\,. \een The
$\gamma_k(z)$'s may be also defined by means of the $\alpha$'s
through the recursion : \ben \gamma_k(z) &=&
\bar\alpha_k\prod_{j=0}^{k-1}
\frac{1-z\widetilde{\gamma}_j(z)}{z-\gamma_j(z)},\\
\label{wt} \widetilde{\gamma}_k(z) &=& \overline{\gamma_k({\bar z}^{-1})}
\,. \een
\end{prop}

\begin{proof}
The first claim is an immediate consequence of (\ref{def:gammaz}).
Now, using $\Phi_k (z) = \prod_{j=0}^{k-1} (z -
\gamma_j(z))$, we obtain
$$\Phi^*_k (z) = \prod_{j=0}^{k-1}(1-z\widetilde{\gamma}_j(z) ),$$
and hence (we use (\ref{def3:gammaz}))
$$\gamma_k(z)=\bar\alpha_k\prod_{j=0}^{k-1}
\frac{1-z\widetilde{\gamma}_j(z)}{z-\gamma_j(z)}.$$
\end{proof}
Note that when $|z|=1$, $|\gamma_k(z)|=|\alpha_k|$. Combined with
the above proposition, this leads us to introduce the following set
of coefficients.
\begin{defn}\label{defdef}
Define the coefficients $(\gamma_k)_{0\leq k\leq n-1}$ by
\begin{equation}\label{deformed}
\gamma_k:=\gamma_k(1),\ \ k=0, \dots, n-1\,.
\end{equation}We shall refer to the $\gamma_k$'s as
the {\it deformed Verblunsky coefficients}.
\end{defn}

\begin{prop}
The following properties hold for the deformed Verblunsky coefficients:
\begin{enumerate}[a)]
\item For all $0\leq k\leq n-1$, $|\gamma_k|=|\alpha_k|$,
and in particular $\gamma_{n-1}\in\mathbb T$;
\item $\gamma_0=\bar{\alpha_0}$ and
\ben
\label{yalpha}
\gamma_k = \bar\alpha_k e^{\ii \varphi_{k-1}}\ \ , \ \
e^{\ii \varphi_{k-1}} = \prod_{j=0}^{k-1}\frac{1-{\bar \gamma}_j}{1-\gamma_j} \ \ , \ \ (k=1, \dots, n-1)\,.
\een
The last term is
special. Since
 $|\alpha_{n-1}| = 1$,  we set
$\alpha_{n-1}= e^{\ii \psi_{n-1}}$, so that \ben\label{last}
\gamma_{n-1} =
e^{\ii (-\psi_{n-1}+ \varphi_{n-2})} := e^{\ii
\theta_{n-1}}\,.\een
\item Let $\mu$ be the spectral measure associated to $(U,e_1)$,
$U\in  U(n)$. Then $\Phi_n(z)$ is  the characteristic polynomial of $U$, in particular,
 \ben
\label{masterf} \Phi_n (1) = \det(\Id-U) = \prod_{k=0}^{n-1} (1 -
\gamma_k). \een
\end{enumerate}
\end{prop}

\begin{proof}
All the results are direct consequences of the definition \ref{defdef}
 and the formulae in Proposition
\ref{prop:gammarecursive} evaluated at $1$.
\end{proof}

\begin{rem}
In \cite{KillipStoiciu}, Killip and Stoiciu have already considered
variables which are the  complex conjugate of our deformed
Verblunsky coefficients as auxiliary variables in the study of the
Pr\"ufer phase (Lemma 2.1 in \cite{KillipStoiciu}). Nevertheless,
the way we define them as well as the  use we make of them  are
different.
\end{rem}

\begin{rem}
The formula (\ref{yalpha}) shows that the $\gamma_k$'s
 can be obtained from the $\alpha_k$'s
 recursively. Hence starting from a spectral measure
associated to a unitary matrix, one can associate with it the
Verblunsky coefficients and then the deformed Verblunsky coefficients.
Conversely, one can translate any property of the deformed ones
 into properties for the spectral measure associated with
it by inverting the transformations  (\ref{yalpha}).
\end{rem}
\begin{rem}
The 
 distribution of the characteristic polynomial
of random unitary matrices evaluated at $1$, through its
Mellin-Fourier transform, plays a key role in the theory
of random matrices, especially through its links with
analytic number theory (see \cite{MezSna} for an account).
In \cite{BHNY}  it is proven that it can be decomposed in law
into a product of independent random variables when working on
the unitary and orthogonal groups endowed with the Haar measure;
since we will prove in Section 3 that the $\gamma_k$'s are independent
under CJ$_{\delta,\beta}^{(n)}$, then we
can conclude that this latter
result holds for any  circular Jacobi ensemble.
\end{rem}


\subsection{Geometric interpretation} 
We give a connection between the coefficients $(\gamma_k)_{0\leq
k\leq n-1}$ and  reflections  defined just below.
This allows us to obtain a new decomposition of the GGT matrix
associated with a measure $\mu$ supported at $n$ points on the
unit circle as a product of $n$ elementary reflections.

Many distinct definitions of reflections on the unitary group exist,
the most well-known may be the Householder reflections. The
transformations which will be relevant to us are the following ones.

\begin{defn} An element $r$ in $U(n)$ will be
referred to as a reflection if $r-\Id$ has rank 0 or 1.
\end{defn}

If $v \in {\mathbb C}^n$,
we denote by $\langle v|$ the linear form $w \mapsto \langle v , w\rangle$.
 The
reflections can also be described in the following way.
If $e$ and $m\not= e$ are unit vectors of $\mathbb C^n$,
there is a unique reflection $r$ such that $r(e) = m$, and
\ben
\label{reflec1}
r = \Id - \frac{1}{1 - \langle m, e\rangle} (m-e)\!\ \langle (m-e) |\ \,.
\een
Let $F:=\span\{e, m\}$ be the 2-dimensional vector
space which is spanned by the vectors $e$ and $m$. It is
clear that the reflection given by formula (\ref{reflec1})
leaves  $F^\perp$ invariant. Now
set
\ben
\label{notref}
\gamma = \langle e, m\rangle \ , \ \rho =
\sqrt{1 - |\gamma|^2} \ , \  e^{\ii \varphi} = \frac{1-\gamma}{1- \bar\gamma}\,,
\een
and let $g\in F$ be the unit vector orthogonal to
$e$ obtained by the Gram-Schmidt procedure. Then in the basis
$(e,g)$ of $F$, the matrix of the restriction of
$r$ is
\ben
\label{defXi}
\Xi(\gamma) := \left(
\begin{array}{cc}
\gamma &\rho e^{\ii \varphi}\\
\rho& -\bar\gamma e^{\ii\varphi}
\end{array}
\right)\,.
\een
Conversely, for $\gamma \in \mathbb D$, such a matrix represents
the unique reflection in $\mathbb C^2$ provided with its canonical
basis, mapping $e_1$ onto $\gamma e_1 + \sqrt{1-|\gamma|^2} e_2$.
The eigenvalues of $r$ are $1$ and $-e^{\ii\varphi}$.

Let $u$ be a unitary  operator in $\mathbb C^n$ and  $e$  a  unit cyclic vector
for $u$. We define $n$ reflections $r_1,\ldots,r_n$ recursively as
follows.  Let $(\varepsilon_1, \dots, \varepsilon_n)$ be the
orthonormal basis obtained from the Gram-Schmidt procedure applied
to $(e, ue, \dots, u^{n-1}e)$.

Let $r_1$ be the reflection, mapping $e= \varepsilon_1$ onto
$u\,e=u \varepsilon_1$.  More generally, for $k\geq 2$ let $r_k$
be the reflection
 mapping $\varepsilon_k$ onto
$r_{k-1}^{-1}r_{k-2}^{-1} \dots r_1^{-1}u\varepsilon_k$. We will
identify these reflections and establish the decomposition of $u$.
Following the basics recalled about GGT matrices in the
introduction, we  note that the matrix of $u$ in the basis
$(\varepsilon_1, \dots, \varepsilon_n)$ is the GGT matrix associated
to the measure $\mu$, i.e. the matrix $\mathcal{G}(\alpha_0,
\cdots , \alpha_{n-2},\alpha_{n-1})$, where $(\alpha_0, \cdots ,
\alpha_{n-2},\alpha_{n-1})$ are the Verblunsky coefficients
associated with the measure $\mu$. We will use formula (4.1.6) of
\cite{SimonBook1} or formula (\ref{rien}) of (our) Lemma
\ref{lemmehessenberg} for the identification of scalar products.

\begin{prop}\begin{enumerate}
\item
For every $1\leq k\leq  n-1$, the reflection $r_k$ leaves invariant
the  $n-2$-dimensional space {\rm Span}$\{\varepsilon_1, \dots,
\varepsilon_{k-1}, \varepsilon_{k+2}, \dots, \varepsilon_n\}$. The
reflection $r_n$  leaves invariant 
Span$\{\varepsilon_1,  \dots, \varepsilon_{n-1}\}$.
\item
The following decomposition holds : \ben u = r_1 \cdots r_n\,. \een
\end{enumerate}
\end{prop}

\begin{proof}
(1) In view of Section 2.1, it is enough to prove that for $j\notin
\{k, k+1\}$, the vectors $\varepsilon_j$ and $r_k\varepsilon_k$ are
orthogonal.

For $k=1$,
$\langle \varepsilon_j, r_1\varepsilon_1\rangle =
\langle \varepsilon_j, u\varepsilon_1\rangle =0$ as soon as $j\geq
3$ from (\ref{rien}).

Assume that for every $\ell \leq k-1$, the reflection
$r_\ell$ leaves invariant\\ 
 Span$\{\varepsilon_1, \dots, \varepsilon_{\ell-1},
\varepsilon_{\ell+2}, \dots, \varepsilon_n\}$. For every $j =1,
\dots, n$, we have \ben \label{clef1} \langle \varepsilon_j,
r_k\varepsilon_k\rangle = \langle
\varepsilon_j,r_{k-1}^{-1}r_{k-2}^{-1} \dots
r_1^{-1}u\varepsilon_k\rangle = \langle r_1 \cdots
r_{k-1}\varepsilon_j, u\varepsilon_k\rangle\,.\een

For $j \geq k+2$, by assumption, the reflections $r_1, \dots,
r_{k-1}$ leave invariant $\varepsilon_j$, so that the above
expression reduces to $\langle\varepsilon_j , u\varepsilon_k\rangle$
which is $0$ again by (\ref{rien}).

For $j=k-1$, we have $r_1\cdots r_{k-1}\varepsilon_{k-1} =
u\varepsilon_{k-1}$ by definition of $r_{k-1}$, so that
(\ref{clef1}) gives $\langle \varepsilon_{k-1},
r_k\varepsilon_k\rangle = \langle u\varepsilon_{k-1},
u\varepsilon_k\rangle$, which is $0$ since $u$ is unitary.

For $j < k-1$, by assumption, the reflections $r_{j+1}, \dots,
r_{k-1}$ leave invariant $\varepsilon_j$, so that the right hand
side of (\ref{clef1})  reduces to $\langle r_1\cdots
r_{j}\varepsilon_j, u\varepsilon_k\rangle$. By definition of $r_j$,
it is $\langle u\varepsilon_j, u\varepsilon_k\rangle$ which is $0$.

(2) For $k$ fixed, it is clear from (1) that $r_1 \cdots r_n
\varepsilon_k = r_1\cdots r_k\varepsilon_k$ which is
$u\varepsilon_k$ by definition of $r_k$.
\end{proof}

\begin{prop}
\label{prop:bingo} For $k= 1, \dots, n-1$, the matrix of the
restriction of $r_k$ to the basis $(\varepsilon_k,
\varepsilon_{k+1})$ is $\Xi(\gamma_{k-1})$  as defined in (\ref{defXi}). In particular \ben
\label{form:bingo} \langle \varepsilon_k , r_k\varepsilon_k\rangle =
\gamma_{k-1}\,. \een The restriction of $r_n$ to $\mathbb C
\varepsilon_n$ is the multiplication by $\gamma_{n-1}$.
\end{prop}

\begin{proof}
Note that for every $k \leq n-1$ \ben \label{identifrho}\langle
\varepsilon_{k+1} , r_k\varepsilon_k\rangle=\langle r_1\cdots
r_{k-1} \varepsilon_{k+1} , u\varepsilon_k\rangle=\langle
\varepsilon_{k+1} , u\varepsilon_k\rangle=\rho_{k-1}\,.\een Since
$r_k$ is a reflection acting on the  subspace \textit{Span}$\{\varepsilon_k,
\varepsilon_{k+1}\}$, identities (\ref{identifrho}) and
(\ref{form:bingo}) entail that the matrix representing $r_k$ in the
basis $(\varepsilon_1,\ldots,\varepsilon_n)$ is precisely
$\Xi(\gamma_{k-1})$ (see (\ref{defXi})). It is then enough to prove
(\ref{form:bingo}).

For $k=1$ it is immediate that:
$$\langle\varepsilon_1,r_1\varepsilon_1\rangle=
\langle\varepsilon_1,u\varepsilon_1\rangle=\bar \alpha_0 =
\gamma_0\,.$$ Let us proceed by induction. For $j \geq 1$ set $q_j
:= \langle \varepsilon_{j}, r_{j} \varepsilon_{j}\rangle$. Assume
that $q_j = \gamma_{j-1}$ for $j\leq k$. We have
 $q_{k+1} = \langle \varepsilon_{k+1}, r_{k+1}
 \varepsilon_{k+1}\rangle = \langle r_1\dots r_k
 \varepsilon_{k+1} , u\varepsilon_{k+1}\rangle$.
Equation (\ref{reflec1}) implies
\ben r_k \varepsilon_{k+1} = \varepsilon_{k+1} - \frac{1}{1- \bar
\gamma_{k-1}}\big(r_k \varepsilon_k - \varepsilon_k \big)\langle
r_k\varepsilon_k , \varepsilon_{k+1}\rangle, \een and since $r_j
\varepsilon_\ell = \varepsilon_\ell$ for $\ell \geq j+2$, we get:
\be r_1\dots r_k \varepsilon_{k+1} = \varepsilon_{k+1} - \frac{1}{1-
\bar \gamma_{k-1}}\big(r_1\dots r_k \varepsilon_k - r_1 \dots
r_{k-1}\varepsilon_k \big) \langle u\varepsilon_k ,
\varepsilon_{k+1} \rangle. \ee Now, it is known that $\langle
u\varepsilon_k , \varepsilon_{k+1}\rangle =
\overline{\langle\varepsilon_{k+1}, u\varepsilon_k\rangle} =
\rho_k$. If we set $v_1 = \varepsilon_1 , $
\[ v_j = r_1 \dots r_{j-1}\varepsilon_j \ , \ a_j =
\frac{\rho_{j-1}}{1- \bar \gamma_{j-1}}\ , \ w_{j+1} =
\varepsilon_{j+1} - a_j u\varepsilon_j\]
 we get the recursion
\ben v_{j+1}= a_j v_j + w_{j+1} \ ,\ (j \leq k)\,, \een which we
solve in : \ben v_{k+1} = \Big(\prod_{j=1}^k a_j\Big)\varepsilon_1 +
\sum_{\ell = 2}^{k+1} \Big(\prod_{j=\ell}^{k} a_{j}\Big)w_\ell. \een
Taking the scalar product with $u\varepsilon_{k+1}$ yields
\[q_{k+1} = \Big(\prod_{j=1}^k \bar a_j \Big)\langle
\varepsilon_1, u \varepsilon_{k+1}\rangle + \sum_{\ell = 2}^{k+1}
\Big(\prod_{j=\ell}^{k} \bar a_{j}\Big) \langle w_\ell ,
u\varepsilon_{k+1}\rangle.\] But $\langle w_\ell , u
\varepsilon_{k+1}\rangle = \langle \varepsilon_\ell, u
\varepsilon_{k+1}\rangle - \bar a_{\ell -1} \langle
u\varepsilon_{\ell - 1}, u \varepsilon_{k+1}\rangle$, and since
$\ell \leq k+1$, we have
\[\langle w_\ell , u \varepsilon_{k+1}\rangle =
\langle \varepsilon_\ell, u \varepsilon_{k+1}\rangle =
-\bar\alpha_{k}\alpha_{\ell-2}\prod_{m=\ell-1}^{k-1} \rho_m,\] which
yields (with $\alpha_{-1} = -1$) \be -\frac{q_{k+1}}{\bar\alpha_k}
&=& \sum_{\ell = 1}^{k+1} \Big(\prod_{j=\ell}^{k} \bar
a_{j}\Big)\alpha_{\ell-2} \prod_{m=\ell-1}^{k-1} \rho_m
 \\&=& \sum_{\ell = 1}^{k+1} \prod_{m=\ell-1}^{k-1}
 \rho_m^2 \prod_{j=0}^{\ell -3}(1 - \bar \gamma_j)
 \frac{\bar \gamma_{\ell -2}}{\prod_{s=0}^{k-1}
 (1 - \gamma_s)} (1-\gamma_{\ell -2}) \\&=& \frac{1}{\prod_{s=0}^{k-1}
 (1 - \gamma_s)}\sum_{\ell = 1}^{k+1}
 \Big[ \prod_{m=\ell-2}^{k-1} \rho_m^2 \prod_{j=0}^{\ell -3}
 (1- \bar \gamma_j) - \prod_{m=\ell-1}^{k-1} \rho_m^2
 \prod_{j=0}^{\ell -2} (1- \bar \gamma_j)\Big]\\
 &=& - \prod_{s=0}^{k-1}
 \frac{ (1 -  \bar \gamma_s)}{(1 - \gamma_s)},\ee
and eventually $q_{k+1} = \gamma_k$.
\end{proof}
Now, we can summarize the above results in the following theorem.

\begin{thm}\label{thm:decoindptes}
Let $u \in U(n)$ 
and $e$ a cyclic vector for $u$.
  Let $\mu$ be the spectral measure of the
  pair $(u,e)$, and  $(\alpha_0,\ldots,
  \alpha_{n-1})$ its Verblunsky coefficients.
  Let $(\varepsilon_1, \dots, \varepsilon_n)$ be
  the orthonormal basis obtained from the Gram-Schmidt
  procedure applied to $(e, ue, \dots, u^{n-1}e)$. Then,
  $u$ can be decomposed as a product of $n$ reflections
  $(r_k)_{1\leq k\leq n}$:
\begin{equation}\label{produitreflexions}
u=r_1\ldots r_n
\end{equation}
where $r_1$ is the reflection mapping $\varepsilon_1$ onto
$u\varepsilon_1$ and by induction for each $2\leq k\leq n$, $r_k$
 maps $\varepsilon_k$ onto
$r_{k-1}^{-1}r_{k-2}^{-1} \dots r_1^{-1}u\varepsilon_k$.

This decomposition can also be restated in terms of  the GGT matrix :
\begin{equation}\label{ggtdecompose}
\mathcal{G}(\alpha_0, \cdots ,\alpha_{n-1})=
\Xi^{(0)}(\gamma_{0})\Xi^{(1)}(\gamma_{1})\ldots\Xi^{(n-1)}(\gamma_{n-1}),
\end{equation}
where for $0\leq k\leq n-2$, the matrix $\Xi^{(k)}$ is given by
 \begin{equation}\label{elementaryreflections}
\Xi^{(k-1)}(\gamma_{k-1})=\Id_{k-1}\oplus \Xi(\gamma_{k-1}) \oplus
\Id_{n-k-1},
\end{equation}
with $\Xi(\gamma)$  defined in (\ref{defXi}). For $k=n-1$,
\begin{equation}\label{elemntaryreflections2}
\Xi^{(n-1)}(\gamma_{n-1})=\Id_{n-1}\oplus(\gamma_{n-1}).
\end{equation}
\end{thm}

\section{Deformed Verblunsky coefficients
and independence}\label{stochastic}

We now use the point of view of sampling (or change of probability
measure) to compute the distribution of the deformed Verblunsky
coefficients under CJ$_{\delta,\beta}^{(n)}$.
Let us first remark that, if the $\alpha_k$'s are independent and with rotational invariant distribution, then from (\ref{yalpha})
\ben\label{invrot}(\alpha_0,\ldots,\alpha_{n-1})\law(\gamma_0,\ldots,
\gamma_{n-1})
\,.\een
This is the case under CJ$_{0,\beta}^{(n)}$.

 We first prove that when $\delta\neq0$
the Verblunsky coefficients are not independent anymore
by studying the simple case $n=2,\beta=2$, 
and then we compute the 
 distribution of
$(\gamma_0,\ldots,\gamma_{n-1})$ under
CJ$_{\delta,\beta}^{(n)}$. We then show that under this
 distribution, 
  the weights of the
measure associated to the Verblunsky coefficients
$(\alpha_0,\ldots,\alpha_{n-1})$ are independent
from the points at which the measure is supported and follow
a Dirichlet distribution.

 Let $\delta \in \mathbb C$ such that $\Re\!\ \delta > -1/2$.  The formula
\begin{eqnarray}\label{measurealter}
\lambda^{(\delta)}(\zeta) &=&
\frac{\Gamma(1+\delta)\Gamma(1+\overline{\delta})}
{\Gamma(1+\delta+\overline{\delta})} (1-\zeta)^{\overline{\delta}}(1-\overline{\zeta})^{\delta}
\ , \ \ \zeta \in \mathbb T
\end{eqnarray}
defines a probability density with respect to the Haar measure on $\mathbb T$,
which  is discontinuous at $1$ when $\Im\,\delta\neq 0$ (see \cite{BNR}).

When $\beta=2$ and $\delta\neq0$, the Verblunsky coefficients are
dependent.  Indeed, let $M\in U(2)$  with Verblunsky coefficients
$\alpha_0$ and $\alpha_1$. Then
$$\det(\Id-M)=[1-\bar{\alpha}_0-\bar{\alpha}_1(1-\alpha_0)],$$
with $|\alpha_0|<1$ and $|\alpha_1|=1$. Under CJ$_{0, 2}^{(2)}$, the variables $\alpha_0$ and $\alpha_1$ are independent and uniformly distributed on $\mathbb D$ and $\mathbb T$ respectively (see \cite{KillipNenciu} or Proposition
\ref{4.2}). The CJ$_{\delta,2}^{(2)}$ is
a $\det_\delta$ sampling of CJ$_{0, 2}\sn$ (see the
Introduction for the definition and notation for the $\det_\delta$-
sampling). So, the joint density of $(\alpha,\varphi)$ on $\mathbb D\times \mathbb T$ is proportional to
\begin{eqnarray*}f(\alpha_0,\alpha_1)&=&[1-\bar{\alpha_0}- \bar{\alpha_1}
(1-\alpha_0)]^{\bar{\delta}}
[1-\alpha_0-\alpha_1(1-\bar{\alpha_0})]^{\delta}\\
&=& (1-\bar{\alpha_0})^
{\bar{\delta}}(1-\alpha_0)^\delta[1-\gamma \bar{\alpha_1}]^
{\bar{\delta}}[1-\bar{\gamma} \alpha_1]^\delta
\,.\end{eqnarray*}
where
$$\gamma=\frac{1-\alpha}{1-\bar{\alpha}}.$$
It is then clear that the
 conditional  density
of $\alpha_1$ given $\alpha_0$ is
\[\alpha_1 \mapsto \lambda^{(\delta)}(\gamma\bar{\alpha_1}) \ \ ,  \ \alpha_1 \in \mathbb T,\]
and this last quantity does not depend on $\alpha_0$ (i.e. on $\gamma$) if and only if $\delta = 0$. Otherwise
 the  Verblunsky
coefficients $\alpha_0$ and $\alpha_1$ are dependent.

The next theorem illustrates our interest in the deformed
Verblunsky coefficients: under CJ$_{\delta,\beta}^{(n)}$,
they are independent.
For the proof of this theorem, we shall
need the following lemma which will also be useful when we study limit theorems:

\begin{lem}\label{lemhypergeo}
Let  $s,t,\ell \in \mathbb C$ such that:
$\Re (s + \ell +1) >0, \Re( t +\ell + 1 )> 0$.
Then, the following identity holds:
\ben
\label{laphua}
\int_{\mathbb D} (1 -|z|^2)^{\ell -1} (1-z)^s (1- \bar z)^t d^2z = \frac{\pi \Gamma(\ell)\Gamma(\ell +1 + s+t)}{\Gamma(\ell+1+s) \Gamma(\ell + 1+t)}\,.
\een
\end{lem}

\begin{proof}
A Taylor expansion yields
\[(1-z)^s (1- \bar z)^t = \sum_{m,n \geq 0}
\rho^{m+n} \frac{(-s)_n (-t)_m}{n! m!}e^{i (m-n)\theta},\]
with  $z = \rho e^{\ii\theta}$ and $0\leq \rho < 1$. We obtain by integrating
\be\int_{\mathbb D} (1 -|z|^2)^{\ell -1} (1-z)^s
(1- \bar z)^t d^2z = 2\pi \sum_{n \geq 0}
\frac{(-s)_n (-t)_n}{n! n!}\int_0^1 (1-\rho^2)^{\ell -1} \rho^{2n+1} d\rho\ee
\be &=& \pi \sum_{n \geq 0} \frac{(-s)_n (-t)_n}{n! }\frac{(\ell -1)!}{(n+\ell)!}
= \frac{\pi}{\ell}\ _2F_1(-s, -t; \ell +1; 1)\,,
\ee
where $\ _2F_1$ is the classical hypergeometric function
(see \cite{AAR})  and an application of Gauss formula
(see \cite{AAR}) shows that the last expression is exactly
the right hand side of (\ref{laphua}).
\end{proof}

\begin{thm}
\label{HPlaw}
Let $\delta\in \mathbb C$ with $\Re\!\ \delta > -1/2$ and
$\beta > 0$. Set $\beta' = \beta/2$.
Under CJ$_{\delta, \beta}\sn$, the distribution of
$(\gamma_0, \dots, \gamma_{n-1})$, denoted hereafter
$\eta_{\delta,\beta}\sn$, is the following:
\begin{enumerate}
\item
the variables  $\gamma_0, \dots, \gamma_{n-2}, \gamma_{n-1}$
 are independent ;
\item for $k=0, \dots, n-2$ the density of $\gamma_k$
with respect to the Lebesgue measure $d^2z$  on $\mathbb C$ is
\[c_{k,n}(\delta) \left(1 -
|z|^2\right)^{\beta'(n-k-1)-1} (1-z)^{\bar \delta}
(1- \bar z)^\delta \mathds{1}_{\mathbb D}(z)\,,\]
where
\ben
\label{valuec} c_{k,n}(\delta) =
\frac{\Gamma\big(\beta'(n-k-1) +1+\delta \big)
\Gamma\big(\beta'(n-k-1) + 1 +\overline{\delta}\big)}
{\pi \Gamma\big(\beta'(n-k-1)\big)
\Gamma\big(\beta'(n-k-1) +1 +\delta + \overline\delta\big)}\,;\een
\item the density  of $\gamma_{n-1}$
with respect to the Haar measure on $\mathbb T$  is $\lambda^{(\delta)}$ (given by (\ref{measurealter})).
\end{enumerate}
\end{thm}
\begin{proof} The distribution of the $\alpha$'s in the
$\beta$-circular unitary ensemble is $\eta_{0, \beta}\sn$. 
 More
precisely, as seen in Definition \ref{defnu} they are independent and if  \[\alpha_k = \textsc{r}_k
e^{\ii \psi_k}\ \ (0\leq k\leq n-2) \ , \ \alpha_{n-1} = e^{\ii \psi_{n-1}}\,, \] then  $\textsc{r}_k$ and
$\psi_k$ are independent, $\psi_k$ is uniformly distributed and
$\textsc{r}_k^2$ has the $\hbox{Beta} (1, \beta'(n-k-1))$
distribution. Moreover  
 $\alpha_{n-1}$ is uniformly distributed on $\mathbb T$.
 From (\ref{masterf}), the sampling factor is
\[\det (\Id-U)^{\bar{\delta}} \det(\Id-\bar U)^{{\delta}} =
(1 - \gamma_{n-1})^{\bar{\delta}} (1- \bar \gamma_{n-1})^{{\delta}}
\prod_{k=0}^{n-2}(1 - \gamma_k)^{\bar{\delta}} (1- \bar \gamma_k)^
{{\delta}}\,.\] so that, under CJ$_{\delta, \beta}\sn$, the density
of $(\textsc{r}_0, \dots, \textsc{r}_{n-2}, \psi_0,
\dots, \psi_{n-1})$ is proportional to
\[
\lambda^{(\delta)}\big(\gamma_{n-1} \big) \prod_{k=0}^
{n-2}(1 - \textsc{r}_k^2)^{\beta'(n-1-k)-1} \textsc{r}_k (1 - \gamma_k)^
{\bar{\delta}} (1- \bar \gamma_k)^{{\delta}}
\mathds{1}_{(0,1)}(\textsc{r}_k) \,,\]
with
\[\gamma_k = \textsc{r}_k e^{\ii \theta_k}\ \ (0\leq k\leq n-2) \ , \ \gamma_{n-1} = e^{\ii \theta_{n-1}}\,. \]
 Thanks  to the relations
(\ref{yalpha}) and (\ref{last}), the Jacobian matrix of the mapping
\[(\textsc{r}_0, \dots, \textsc{r}_{n-2}, \psi_0,
\dots, \psi_{n-1}) \rightarrow (\textsc{r}_0, \dots,
\textsc{r}_{n-2}, \theta_0, \dots, \theta_{n-1})\] is
lower triangular with diagonal elements $\pm 1$, so that, under
CJ$_{\delta, \beta}\sn$, the density of $(\textsc{r}_0, \dots,
\textsc{r}_{n-2}, \theta_0, \dots, \theta_{n-1})$, is proportional
to  \ben \label{densy}\lambda^{(\delta)}(\gamma_{n-1})
 \prod_{k=0}^{n-2}(1 - \textsc{r}_k^2)^{\beta'(n-1-k)-1}
 \textsc{r}_k (1 - \gamma_k)^{\bar{\delta}} (1- \bar \gamma_k)
 ^{{\delta}} \mathds{1}_{(0,1)}(\textsc{r}_k) \,, \een
which proves the required independence and the expression of
the distributions, up to the determination of $c_{k,n}(\delta)$. This quantity
 is obtained by taking $\ell = \beta'(n-k-1), s
= \overline{\delta}, t = \delta$ in (\ref{laphua}), which gives
(\ref{valuec}) and completes the proof of the Theorem.
\end{proof}
Starting with a set of deformed Verblunsky coefficients, with distribution $\eta_{\delta, \beta}^{(n)}$,
 we obtain
the  coefficients  $(\alpha_0,\ldots,
\alpha_{n-1})$ by inverting formula (\ref{masterf}).
These  are the coordinates of  some probability
measure $\mu$ supported at $n$ points on the unit circle:
$$\mu=\sum_{k=1}^{n}\pi_{k}\delta_{e^{\ii\theta_k}},$$with
$\pi_k>0$ and $\sum_{k=1}^n\pi_k=1$. The next theorem gives
the  distribution induced on the vector
$(\pi_1,\ldots,\pi_{n},\theta_1,\ldots,\theta_n)$ by
$(\gamma_0,\ldots,\gamma_{n-1})$.

\begin{thm}\label{generalisationKN4.2}
 The following formulae express the same measure on the manifold
 of probability distribution on $\mathbb T$ supported at $n$ points:
$$K_{\delta,\beta}^{(n)}|
\Delta(e^{\ii\theta_1},\dots,e^{\ii\theta_n})|^\beta
\prod_{k=1}^{n}(1-e^{-\ii\theta_k})^{\delta}(1-e^{\ii\theta_k})^
{\overline{\delta}}\prod_{k=1}^n\pi_k^{\beta'-1 }\dd\theta_1
\ldots\dd\theta_n\dd\pi_1\ldots\dd\pi_{n-1}$$
in the $(\theta,\pi)$ coordinates and
$$K_{\delta,\beta}^{(n)}\prod_{k=0}^{n-2}(1-|\gamma_k |^2)^
{\beta'(n-k-1)-1}\prod_{k=0}^{n-1}(1 - \gamma_k)^{\bar{\delta}}
(1- \bar \gamma_k)^{{\delta}}d^2\gamma_0\ldots d^2\gamma_{n-2}
d\phi$$ in terms of the deformed Verblunsky coefficients, with
$\gamma_{n-1}=e^{\ii\phi}$. Here, $K_{\delta,\beta}^{(n)}$ is a constant:
$$K_{\delta,\beta}^{(n)}=\frac{\Gamma(1+\delta)
\Gamma(1 +\bar\delta)}{2^{n-1}\pi\Gamma(1 + \delta+\bar\delta)}
\prod_{k=0}^{n-2}c_{k,n}(\delta),$$with $c_{k,n}(\delta)$ given in
Theorem \ref{HPlaw}. Consequently, if $(\gamma_0, \dots,
\gamma_{n-1})$ is $\eta_{\delta,\beta}\sn$ distributed,
 then 
 $(\pi_1,\ldots,\pi_n)$ and
  $(\theta_1,\ldots,\theta_n)$ are independent; the vector of
 weights $(\pi_1,\ldots,\pi_n)$ follows the Dir$_n(\beta')$ distribution
 and the vector $(\theta_1,\ldots,\theta_n)$ has the density $h_{\delta, \beta}\sn$.

\end{thm}
\begin{proof}
In the course of this proof, we shall adopt the following point
of view. Starting with a measure supported at $n$ points on the
unit circle, we associate with it its Verblunsky coefficients
$(\alpha_0,\ldots,\alpha_{n-1})$ and then the corresponding
 GGT matrix which we note $G$ for
simplicity. Then $e_1$ is a cyclic vector for $G$ and $\mu$ is
the spectral measure of $(G,e_1)$. Conversely, starting with the
set of deformed Verblunsky coefficients with $\eta_{\delta, \beta}$ distribution, 
 we construct the coefficients
$(\alpha_0,\ldots,\alpha_{n-1})$ with the
transformations (\ref{yalpha}), then
 the GGT
matrix associated with it and 
finally  $\mu$ the spectral
measure associated with this matrix and $e_1$.

We use the following well-known identity
(see \cite{SimonBook1} or \cite{KillipNenciu} Lemma 4.1):
\begin{equation}\label{j1}
|\Delta(e^{\ii\theta_1},\dots,e^{\ii\theta_n})|^2
\prod_{k=1}^n\pi_k=\prod_{k=0}^{n-2}(1-|\alpha_k|^2)^{n-k-1}.
\end{equation}
Since $|\gamma_k|=|\alpha_k|$, we can also write
\begin{equation}\label{j2}
|\Delta(e^{\ii\theta_1},\dots,e^{\ii\theta_n})|^2
\prod_{k=1}^n\pi_k=\prod_{k=0}^{n-2}(1-|\gamma_k|^2)^{n-k-1}.
\end{equation}
Moreover, from (\ref{masterf}),
\begin{equation}\label{j3}
\det(\Id-G)=\prod_{k=1}^{n}(1-e^{\ii\theta_k})=\prod_{k=0}^{n-1}(1-\gamma_k).
\end{equation}
In our setting, $\pi_i$ is modulus squared of the first component of the $i$-th  eigenvector of the matrix $G$. Now, define
$$q_k^2=\pi_k,\;\text{ for } k=1,\dots,n.$$
It is known (see for example Forrester \cite{ForrBook}, Chapter 2
and \cite{ForresterRains} Theorem 2) that the Jacobian of the map
$(\alpha_0,\ldots,\alpha_{n-1})\mapsto(\theta_1,\ldots\theta_n,q_1,\ldots,q_{n-1})$
is given by
$$\frac{\prod_{k=0}^{n-2}(1-|\alpha_k|^2)}{q_n\prod_{k=1}^nq_k}.$$
Moreover, the map
$(\gamma_0,\ldots,\gamma_{n-1})\mapsto(\alpha_0,\ldots,\alpha_{n-1})$
is invertible and its Jacobian is $1$, as already seen. The result
now follows from simple integral manipulations combined with the
identities (\ref{j2}) and (\ref{j3}).
\end{proof}

\section{Matrix models  for the
circular Jacobi ensemble}\label{secmatrixmodel}
The results of the previous sections can now be used to propose some
simple matrix models for the  
circular Jacobi ensemble. There are
mainly two ways to generate matrix models for a given spectral
measure encoded by its Verblunsky coefficients.

{\it The AGR decomposition  :} if
$U=\Theta^{(0)}(\alpha_0)\Theta^{(1)}(\alpha_1)\dots\Theta^{(n-1)}(\alpha_{n-1})$
(the $\Theta_k$'s are defined in the introduction),
 the Verblunsky coefficients for the spectral measure
 associated to $(U,e_1)$ are precisely $(\alpha_0,\dots,\alpha_{n-1})$
 (see \cite{AGR} or \cite{Simon5years} Section 10).
Therefore, taking independent $\alpha_k$'s with law $\eta_{0, \beta}^{(n)}$,
 the density
 of the eigenvalues  of
$$U=\Theta^{(0)}(\alpha_0)\Theta^{(1)}(\alpha_1)\dots
\Theta^{(n-1)}(\alpha_{n-1})$$ is proportional to
$|\Delta(e^{\ii\theta_1},\dots,e^{\ii\theta_n})|^\beta$. The matrix
$U$ obtained above is the GGT matrix associated with the $\alpha_k$'s.
 It is in the Hessenberg form.

{\it The CMV form :} set
$$
\left\{
\begin{array}{ccc}
\mathcal{L}&=&\Theta^{(0)}(\alpha_0)\Theta^{(2)}(\alpha_2)\dots\\
\mathcal{M}&=&\Theta^{(1)}(\alpha_1)\Theta^{(3)}(\alpha_3)\dots
\end{array}
\right.
$$
Cantero, Moral, and Velazquez \cite{CMV} proved that
the Verblunsky coefficients associated to $(\mathcal{L}\mathcal{M},e_1)$
are precisely $(\alpha_0,\dots,\alpha_{n-1})$.
Therefore, taking as previously independent $\alpha_k$'s
with distribution $\eta_{0, \beta}^{(n)}$,
the density of the eigenvalues of
the spectral law of
$\mathcal{L}\mathcal{M}$
is proportional to $|\Delta(e^{\ii\theta_1},\dots,e^{\ii\theta_n})|^\beta$ (\cite{KillipNenciu}).
This matrix model is very sparse:
it is pentadiagonal.

We now propose a matrix model  for the
circular Jacobi ensemble: it is reminiscent of the AGR factorization with
the noticeable difference that it is based on the deformed
Verblunsky coefficients and actual reflections as defined in Section 2.
\begin{thm}\label{thm:new matrix model}
If $(\gamma_0, \dots, \gamma_{n-1})$ is $\eta_{\delta,\beta}\sn$
distributed, then with the notation of (\ref{elementaryreflections})
and  (\ref{elemntaryreflections2}),
$$\Xi^{(0)}(\gamma_0)\Xi^{(1)}(\gamma_1)\dots\Xi^{(n-1)}
(\gamma_{n-1})$$ is a matrix model for the
circular Jacobi ensemble,
i.e. the density of the eigenvalues  is $h_{\delta, \beta}^{(n)}$ (see (\ref{loiHP})).
\end{thm}
\begin{proof}
We know from Theorem \ref{thm:decoindptes} that
\begin{equation}
\mathcal{G}(\alpha_0, \cdots ,\alpha_{n-1})
=\Xi^{(0)}(\gamma_{0})\Xi^{(1)}(\gamma_{1})\ldots\Xi^{(n-1)}(\gamma_{n-1}).
\end{equation}
We also proved in Theorem \ref{generalisationKN4.2} that
the set of deformed Verblunsky coefficients with probability
distribution $\eta_{\delta, \beta}^{(n)}$ 
 induces
a distribution on the eigenvalues of the GGT matrix
$\mathcal{G}(\alpha_0, \cdots ,\alpha_{n-1})$
which has exactly the density $h_{\delta, \beta}$. 
This completes the proof of the Theorem.
\end{proof}
\begin{rem}
We now say a few extra words on the CMV form obtained by
Killip and Nenciu in \cite{KillipNenciu}. Cantero, Moral
and Velazquez \cite{CMV} introduced the basis
$\chi_0,\ldots,\chi_{n-1}$ obtained by orthogonalizing
the sequence $1,z,z^{-1},\ldots$. They prove that in this
basis the 
 matrix is pentadiagonal. We name
this matrix $\mathcal{C}(\alpha_0, \cdots ,\alpha_{n-1})$.
It turns out that there exists a unitary 
 $P$ 
such that:
$$P\mathcal{G}(\alpha_0, \cdots ,\alpha_{n-1})P^\star=
\mathcal{C}(\alpha_0, \cdots ,\alpha_{n-1})\; , \;
P\varphi_0=\chi_0.$$
The two pairs $(\mathcal{G}(\alpha_0, \cdots ,\alpha_{n-1}),\varphi_0)$ and  $(\mathcal{C}(\alpha_0,
\cdots ,\alpha_{n-1}),\chi_0)$ are equivalent, they admit
the $\alpha_k$'s as Verblunsky coefficients, and have the same
spectral measure. We conclude that if we start with the $\gamma_k$'s distributed as $\eta_{\delta, \beta}^{(n)}$,
 and build the $\alpha_k$'s by inverting
the transformation (\ref{yalpha}), then $\mathcal{C}(\alpha_0, \cdots ,\alpha_{n-1})$ will be a matrix model for the
circular Jacobi
 ensemble.  But we do not know how to construct the CMV matrix
from the $\gamma_k$'s directly. We saw at the beginning of this section
that Cantero et al. introduced the matrices $\mathcal{L}$ and
$\mathcal{M}$, as direct product of small blocks $\Theta^{(k)}(\alpha_k)$
and obtained $\mathcal{C}$ as $\mathcal{C}=\mathcal{L}\mathcal{M}$.
It would be interesting to have an analogue construction
based on the independent $\gamma_k$'s.
\end{rem}

Theorem \ref{thm:decoindptes} which is a deterministic result, has
also the following consequence:
\begin{prop}\label{lem:elementary reflections}
Let $(\alpha_0,\dots,\alpha_{n-2},\alpha_{n-1})\in\DD^{n-1}\times\mathbb T$
be independent random variables
with rotationally invariant distribution. Then
$$
\Theta^{(0)}(\alpha_0)\Theta^{(1)}(\alpha_1)\dots
\Theta^{(n-1)}(\alpha_{n-1})\law
\Xi^{(0)}(\alpha_0)\Xi^{(1)}(\alpha_1)\dots\Xi^{(n-1)}(\alpha_{n-1}).
$$
\end{prop}

\begin{proof}
We give two proofs of this result. The first one
is a consequence of Theorem \ref{thm:decoindptes}
from which we know that
$$\Theta^{(0)}(\alpha_0)\Theta^{(1)}(\alpha_1)\dots
\Theta^{(n-1)}(\alpha_{n-1})=\Xi^{(0)}(\gamma_{0})
\Xi^{(1)}(\gamma_{1})\ldots\Xi^{(n-1)}(\gamma_{n-1})\,,$$
and the remark at the beginning of Section \ref{stochastic}.

For the second proof, we proceed by induction on $n$. For $n=1$ the
result is obvious. Suppose the result holds at rank $n-1$ : thanks
to the recurrence hypothesis,
$$
\Xi^{(0)}(\alpha_0)\Xi^{(1)}(\alpha_1)\dots\Xi^{(n-1)}(\alpha_{n-1})\law
\Xi^{(0)}(\alpha_0)\Theta^{(1)}(\alpha_1)\dots\Theta^{(n-1)}(\alpha_{n-1})
.
$$
Let $e^{\ii\phi_0}=\frac{1-\overline{\alpha_0}}{1-\alpha_0}$.
An elementary calculation gives
\begin{multline*}
\Xi^{(0)}(\alpha_0)\Theta^{(1)}(\alpha_1)\dots\Theta^{(n-2)}
(\alpha_{n-2})\Theta^{(n-1)}(\alpha_{n-1})\\
=\Theta^{(0)}(\alpha_0)\Theta^{(1)}(e^{-\ii\phi_0}\alpha_1)
\dots\Theta^{(n-2)}(e^{-\ii\phi_0}\alpha_{n-1})\Theta^{(n-1)}
(e^{-\ii\phi_0}\alpha_{n-1}).
\end{multline*}

As the $\alpha_k$'s are independent with law invariant by rotation,
$$(\alpha_0,e^{-\ii\phi_0}\alpha_1,\dots,e^{-\ii\phi_0}
\alpha_{n-2},e^{\ii\phi_0}\alpha_{n-1})\law
(\alpha_0,\alpha_1,\dots,\alpha_{n-2},\alpha_{n-1}),$$
which completes the proof.
\end{proof}
Now that we have a matrix model for the
circular Jacobi ensemble, we can study the characteristic polynomial for
such matrices; the key formula will be (\ref{masterf}).

\begin{prop}\label{detJac}
Let $U$ be  a unitary matrix of size $n$ and let $Z_n = \det (\Id-U)$ be its characteristic
polynomial evaluated at $1$.  Then, in the   
circular Jacobi ensemble, $Z_n$ can be written as a product
of $n$ independent complex random variables:
\[Z_n = \prod_{k=0}^{n-1}(1-\gamma_k)\,,\]
where the laws of the $\gamma_k$'s are  given in Theorem
\ref{HPlaw}. Consequently
for any $s,t\in\CC$, with $\Re(t)>-\frac{1}{2}$,
 the Mellin-Fourier transform of
 $Z_n$ is :
\begin{eqnarray}\label{MFT}
\nonumber \E[|Z_n|^te^{\ii s\arg Z_n}] = \ \ \ \ \ \ \ \ \ \ \ \ \ \ \ \ \ \ \ \ \ \ \ \ \ \\
\prod_{k=0}^{n-1}
\dfrac{\Gamma(\beta'k+1+\delta)
\Gamma(\beta'k+1+\bar{\delta})\Gamma(\beta'k+1+\delta+\bar{\delta}+t)}
{\Gamma(\beta'k+1+\delta+\bar{\delta})\Gamma(\beta'k+1+\delta+\frac{t-s}{2})
\Gamma(\beta'k+1+\bar{\delta}+\frac{t+s}{2})}.
\end{eqnarray}
\end{prop}
\begin{proof}
The first part is an easy consequence of (\ref{masterf}) and
Theorem \ref{HPlaw}. To prove the second part, we note that
if $X_k=(1-\gamma_k)$, then $|X_k|^te^{\ii s\arg X_k}=(1-\gamma_k)^a
(1-\bar{\gamma}_k)^b$, where $a=(t+s)/2$ and $b=(t-s)/2$.
Consequently, by independence of the $\gamma_k$'s, we obtain:
$$\E[|Z_n|^te^{\ii s\arg Z_n}]=\prod_{k=0}^{n-1}\E[(1-\gamma_k)^a
(1-\bar{\gamma}_k)^b],$$and  formula (\ref{MFT}) then easily
follows from Lemma \ref{lemhypergeo}.
\end{proof}

\section{Limiting spectral measure and large deviations}

In (\ref{defmuw}) we defined the spectral measure which is a central tool for the study of our circular ensembles.  Let us re-write this measure on $\mathbb T$ as
\ben
\mu_\w\sn := \sum_{k=1}^n \pi_k\sn \delta_{e^{\ii \theta_k\sn}},
\een
where
we put a superscript $\sn$ to stress on the dependency on $n$. Besides, in  classical Random Matrix Theory, many authors are  mainly
 interested  in the empirical spectral distribution (ESD) defined by
\ben
\mu_\u\sn = \frac{1}{n} \sum_{k=1}^n \delta_{e^{\ii \theta_k\sn}}\,.
\een
We are concerned with  their asymptotics
under  CJ$_{\delta, \beta}\sn$
 when $n \rightarrow \infty$ with \[\delta = \delta(n) = \beta'n \dd\,\]
 where $\Re\!\ \dd \geq 0$ (and as usual $\beta'=\beta/2$). In this framework
the variables $(\theta_1\sn, \dots, \theta_n\sn)$ and $(\pi_1\sn, \dots, \pi_n\sn)$ are distributed as in Theorem  \ref{generalisationKN4.2}.
It is well-known that in the CUE (i.e. for $\dd =0$ and $\beta' =1$), the sequence $(\mu_\u\sn)$ converges weakly in probability to the uniform measure on $\mathbb T$ and
that it satisfies the Large Deviation Principle at scale $n^2$ (\cite{HiaiP} chap. 5).

In this section we prove that in the cJUE,  both sequences $(\mu_\w\sn)$ and $(\mu_\u\sn)$ converge weakly in probability to the same limit measure supported by an arc of the unit circle, and
that $(\mu_\u\sn)$ satisfies the Large Deviation Principle at scale $n^2$.
In the next two subsections, we first recall some definitions (which may be skipped by probabilist readers) and then state the main results which are Theorems \ref{cvsmalter}, \ref{cvesdalter} and \ref{LDPESDalter}.

\subsection{Weak convergence of measures}
\label{recallwc}
Let ${\mathcal M}_1 (\mathbb T)$ be the space of probability measures on $\mathbb T$. The weak topology on ${\mathcal M}_1(\mathbb T)$ is defined from the duality with the space of continuous functions on $\mathbb T$; that is, $\nu_n$ weakly converges to $\nu$ means $\int f d\nu_n \rightarrow \int f d\nu$ for every continuous function $f$. Since $\mathbb T$ is compact, it is equivalent to the convergence of moments.
If we define the distribution function of $\nu \in {\mathcal M}_1(\mathbb T)$ as   $F_{\nu} (t) = \nu (\{e^{\ii \theta} ; \theta \in [0, t) \})$, then  $\nu_n \rightarrow \nu$ weakly if and only if
\[F_{\nu_n}(t) \rightarrow F_{\nu}(t)\]
for all $t \in [0, 2\pi)$ at which $F_\nu (t)$ is continuous. The L\'evy distance\footnote{whose precise definition is not needed here (see \cite{DZ} Theorem D8)} $d_L$ is compatible with this topology  and makes ${\mathcal M}_1(\mathbb T)$ a compact metric space.
It should be noticed that for all pair of elements of ${\mathcal M}_1(\mathbb T)$
\begin{equation}
\label{Levy}
d_L(\mu, \nu) \leq \sup_t |F_\mu(t) - F_\nu (t)|\,.
\end{equation}

In the following we consider random probability measures $(\nu_n)$ and the weak convergence in probability of $(\nu_n)$ to a deterministic $\nu\in {\mathcal M}_1(\mathbb T)$. That means that for every $n \geq 1$ we have a probability space $(\Omega_n, {\mathcal F}_n , Q_n)$, a ${\mathcal M}_1(\mathbb T)$-valued random variable  $\nu_n$, and that $d(\nu_n , \nu) \rightarrow 0$ in probability, where $d$ is any distance compatible with the weak topology (e.g. the L\'evy distance $d_L$). In other words
\begin{eqnarray}\nonumber \lim_n \nu_n = \nu \ \ \hbox{(in probability)} \ \ &\Leftrightarrow& \ \lim_n d(\nu_n , \nu) = 0 \ \ \hbox{(in probability)}\\ \nonumber
&\Leftrightarrow& \ \ \forall \varepsilon > 0 \ \ \ \lim_n Q_n (d(\nu_n , \nu) > \varepsilon)=0\,.\\ \label{wL}  \end{eqnarray}
\subsubsection{The spectral measure}
The following theorem states the convergence of  the sequence $(\mu_\w\sn)$  to an  explicit limit, which we identify now, with the help of  some more notation. We assume that $\Re\!\ \dd \geq 0$. Let
\ben\label{alphad}
\alpha_\dd = -\frac{\overline\dd}{1+\overline\dd}
\een
and let $\theta_\dd \in [0, \pi)$ and $\xi_\dd \in [-|\theta_\dd|, |\theta_\dd|]$ be such that
\[\sin\frac{\theta_\dd}{2} = \left|\frac{\dd}
{ 1 + \dd}\right| \ \ , \ \   e^{\ii \xi_\dd} = \frac{1 + \dd}{1+\overline\dd}\,.\]
If $\Re\!\ \dd \geq 0, \dd \not =0$, let  $w_\dd$ be defined by
\begin{equation}\label{wd}
w_\dd (\theta) =
\begin{cases} \displaystyle\frac{\sqrt{\sin^2\big((\theta-\xi_\dd)/2\big) -
\sin^2(\theta_\dd/2)}}{|1+\alpha_\dd| \!\ \sin(\theta/2)}
 & \text{if $\theta \in (\theta_\dd+\xi_\dd, 2\pi-\theta_\dd+\xi_\dd)$}
\\
0 &\text{otherwise}.
\end{cases}
\end{equation}
In this case let
\begin{equation}
\label{identifmu}
d\mu_\dd (\zeta) := w_\dd (\theta) \frac{d\theta}{2\pi} \ \  \ (\zeta = e^{\ii \theta})\,,\end{equation}
and for $\dd =0$ let
\[d\mu_0 (\zeta) := \frac{d\theta}{2\pi} \ \  \ (\zeta = e^{\ii \theta})\,,\]
the Haar measure on $\mathbb T$.
\begin{thm}\label{cvsmalter}
Assume $\Re\!\ \dd \geq 0$. As $n \rightarrow \infty$,
\[\lim \mu_\w\sn = \mu_\dd \ \ \hbox{(in probability)}\,.\]
\end{thm}

To prove this convergence we use the parameterization of measures by their
modified Verblunsky coefficients and the  following two lemmas whose
proofs are postponed until the end of the section.

\begin{lem}
\label{lemlimy}
For every fixed $k \geq 0$, as $n \rightarrow \infty$,
\ben
\label{limy}
\lim \gamma_k\sn= -\frac{\dd}{1+\overline\dd} \ \ \hbox{(in probability)}\,, \een
and consequently
\begin{equation}
\label{limalpha}
\lim \alpha_k\sn = \alpha_\dd e^{-\ii (k+1) \xi_\dd} \ \ \hbox{(in probability)}
\,.\end{equation}
\end{lem}

\begin{lem}
\label{nouveaulem}
The sequence of Verblunsky coefficients of 
$\mu_\dd$ is precisely   $\left(\alpha_\dd e^{-\ii (k+1) \xi_\dd}\right)_{k \geq 0}$.
\end{lem}
\begin{proof}[Proof of Theorem 5.1]
For $\nu \in {\mathcal M}_1(\mathbb T)$, let $m_k (\nu) = \int e^{\ii k\theta} d\nu(\theta)$ be its $k$-th moment and let $\alpha_k(\nu)$ be its $k$-th Verblunsky - in short V- coefficient. Moments are related to V-coefficients in a continuous way : for every $j\geq 1$, $m_j(\nu)$ is a continuous function of
$(\alpha_0, \dots, \alpha_{(j\wedge N)-1})$ where $N$ is the cardinal of the support of $\nu$.
%
We know for Lemma \ref{lemlimy} that
for every fixed $k$, $\lim \alpha_j(\mu_\w\sn) = \alpha_j(\mu_\dd)$
 for  $j \leq k$, in probability.
It entails the convergence of $m_k(\mu_\w\sn)$ to $m_k(\mu_\dd)$,
 in probability, hence the weak convergence of the measures.
\end{proof}
\medskip

\begin{rem}
 The above method can easily be adapted to show the
trivial asymptotics of two different scaling regimes :
\begin{itemize}
\item if $\delta(n)=o(n)$, 
the limit is  the uniform measure on $\mathbb T$,
\item if $\delta(n)$ is real and $n=o(\delta(n))$, 
the limit is the Dirac measure at  $-1$.
\end{itemize}
\end{rem}
\medskip

\begin{proof}[Proof of Lemma \ref{lemlimy}]
For $\gamma_k\sn$ we use the Mellin transform
\[\mathbb E \left( (1- \gamma_k\sn)^s \right) = \frac{\Gamma
\big(\beta'(n-k-1)+\delta +\overline\delta+s+1\big)
\Gamma\big(\beta'(n-k-1) +\overline\delta+1\big)}
{\Gamma\big(\beta'(n-k-1)+\delta +\overline\delta+1\big)
\Gamma\big(\beta'(n-k-1) +\overline\delta+s+1\big)},\]
(this comes immediately from (\ref{laphua})).

Since for fixed $z\in \mathbb C$
\[\lim_{n\rightarrow \infty} \frac{\Gamma(n+z)}
{\Gamma(n)n^z} =1,\]
we get, for fixed $s$ (and $k$)
\[\lim_{n\rightarrow \infty} \mathbb E \left((1- \gamma_k\sn)^s \right)=
\left(\frac{1+\overline\dd +\dd}{1+\overline\dd}\right)^s,\]
which implies that
\[\lim_{n\rightarrow \infty} (1-\gamma_k\sn) = \frac{1+\overline\dd +\dd}
{1+\overline\dd} \ \ \hbox{in probability}\,,\]
 which is equivalent to (\ref{limy}).
The statement (\ref{limalpha}) is a direct consequence of
(\ref{limy}) and (\ref{yalpha}).
\end{proof}
\begin{rem}
The convergences in probability in the above lemma actually
hold in $L^p$, for all $p>0$ because all variables are bounded by $1$.
\end{rem}
\medskip

\begin{proof}[Proof of Lemma \ref{nouveaulem}] If $\dd$ is real, the Verblunsky coefficients are all equal. In fact,
 measures satisfying this property are known to have an absolute continuous part supported by an arc of the unit circle, and a possible additional Dirac mass (\cite{SimonBook1} p.87). More precisely, let $\alpha \in \mathbb D$ and let $\nu$ be the measure on
$\mathbb T$ such that $\alpha_k(\nu) = \alpha$ for every $k \geq 0$.
 If  $\theta(\alpha)$ and $\xi(\alpha)$ are defined by
\[\theta (\alpha) = 2 \arcsin |\alpha| \ , \
 e^{\ii \xi(\alpha)} = \frac{1+\alpha}
 {1+\overline\alpha}\,,\]
then
\[d\nu(\zeta)  = w(\theta) \frac{d\theta}{2\pi} \ \  \ (\zeta = e^{\ii \theta})\,,\]
where
\begin{equation}\label{wd}
w (\theta) =
\begin{cases} \displaystyle\frac{\sqrt{\sin^2(\theta/2\big) -
\sin^2(\theta(\alpha)/2)}}{|1+\alpha| \!\ \sin((\theta+\xi(\alpha))/2)}
 & \text{if $\theta \in (\theta(\alpha), 2\pi-\theta(\alpha)$}
\\
0 &\text{otherwise}
\end{cases},
\end{equation}
as soon as
\ben\label{add}|\alpha + \frac{1}{2}| \leq \frac{1}{2}\een
(otherwise, there is an additional Dirac mass).
The orthogonal polynomials with respect to this measure are known as the "Geronimus polynomials".

When $\alpha= \alpha_\dd$, we set $\theta(\alpha)=\theta_\dd$, $\xi(\alpha) = \xi_\dd$
and $\nu = \nu_\dd$, $w= w_\dd$.
 We see that
\[\alpha_\dd + \frac{1}{2} = \frac{1-\overline\dd}
{2(1+\overline\dd)}\,,\]
so that the condition (\ref{add}) is fulfilled if and only
if $\Re\!\ \dd \geq 0$, which we assumed.

Moreover, it is known (see \cite{SimonBook2} p.960) that if $(\alpha_k)_{k\geq0}$ is the sequence of Verblunsky coefficients of a measure $\mu$, then the  coefficients $(e^{-\ii (k+1)\xi_\dd}\alpha_k)_{k\geq0}$ are associated with  $\mu$ rotated  by
$\xi_\dd$. Consequently,
\[d\mu_\dd (\zeta) = d\nu_\dd (\zeta e^{-\ii\xi_\dd})\,,\]
which is precisely (\ref{identifmu}).
\end{proof}

\subsubsection{The ESD}
 In matrix models, the convergence of the ESD is often tackled directly via the convergence of moments or the convergence of the Cauchy transform.
 Here, we follow a different way:
 we use Theorem \ref{cvsmalter}
and prove that the two sequences  $(\mu_\w\sn)_n$ and $(\mu_\u\sn)_n$ are
"contiguous".

\begin{thm}\label{cvesdalter}
Assume $\Re\!\ \dd \geq 0$. As $n \rightarrow \infty$,
\[\lim_n \mu_\u\sn = \mu_\dd \ \ \hbox{(in probability)}\,,\]
where $\mu_\dd$ is  given in (\ref{identifmu}).
 \end{thm}

\begin{proof}   From Theorem 5.1,
we know, with the notation from the beginning of Section \ref{recallwc} that
\[\lim_{n\rightarrow \infty} d_L (\mu_\w\sn , \mu_\dd)  =0 \ \ \hbox{in probability}\,.\]
By the triangle inequality
\[d_L (\mu_\u\sn , \mu_\dd)\leq d_L (\mu_\w\sn , \mu_\dd) + d_L (\mu_\w\sn , \mu_\u\sn)\,\]
so, it is enough to prove that $d_L (\mu_\w\sn , \mu_\u\sn)$ converges to $0$ in probability.
 Thanks to (\ref{Levy}), it is enough to prove
\begin{equation}\label{toshowalter}
\lim_{n \rightarrow \infty} \sup_t |F_{\mu_\w\sn}(t) - F_{\mu_\u\sn}(t)| = 0 \ \ \hbox{in probability}\,.
\end{equation}
In fact,
\begin{equation}\label{supfr}\sup_t |F_{\mu_\w\sn}(t) - F_{\mu_\u\sn}(t)| = \max_k \big|S_k\sn - \frac{k}{n}\big|\,,
\end{equation}
where $S_k\sn = \sum_{j=1}^k \pi_j\sn$.
 Using  the union bound and the Markov inequality, we may write
\begin{eqnarray}
\nonumber
\Prob\left(\max_k\left|S_k\sn-\frac{k}{n}\right|>\varepsilon\right)\leq
\sum_{k=1}^n\Prob\left(\left|S_k\sn-\frac{k}{n}\right|>\varepsilon\right)\\
\label{Markov} \leq
\varepsilon^{-4} \sum_{k=1}^n \E\left(\left(S_k\sn-\frac{k}{n}\right)^4\right)\,,
\end{eqnarray}
and then use explicit distributions. Indeed,
we showed in Theorem \ref{generalisationKN4.2} that the
vector $(\pi_1\sn, \dots, \pi_n\sn)$ follows the
Dir$_n(\beta')$ distribution. %
 It entails that for  $k = 1, \dots, n -1$,
the variable $S_k\sn$
 is  Beta$(\beta' k, \beta'(n-k))$
distributed.

Recall that the Mellin transform of a beta variable Beta$(a,b)$ with positive parameters $a$ and $b$ is
\[
\E\left(\text{Beta}(a,b)^s\right)=\frac{\Gamma(a+s)\Gamma(a+b)}{\Gamma(a)\Gamma(a+b+s)}\,,
\]
for $s > -a$. As a consequence, $\E \text{Beta}(a,b) = a/(a+b)$ and a straightforward calculation gives
\be
\E\left(\left(\text{Beta}(a,b)-\E\text{Beta}(a,b))\right)^4\right)&=&\frac{3ab(2a^2+2b^2-2ab+a^2b+ab^2)}{(a+b)^4(a+b+1)(a+b+2)(a+b+3)}\\
&=&O\left(\frac{ab}{(a+b)^4}\right).
\ee
When $a=\beta'k$ and $b=\beta'(n-k)$, this shows that
 the $k$-th term in the sum of (\ref{Markov}) is $O\left(\frac{k(n-k)}{n^4}\right)$, so that
\ben
\nonumber
\Prob\left(\max_k\left|S_k\sn-\frac{k}{n}\right|>\varepsilon\right)
=O\left(\sum_{k=1}^n\frac{k(n-k)}{n^4}\right)=O\left(\frac{1}{n}\right)\,.
\een
Thanks to (\ref{supfr}), this yields (\ref{toshowalter}) and completes the proof.
\end{proof}
\subsection{Large deviations for the ESD}
It turns out that the  convergence in Theorem \ref{cvesdalter} is exponentially fast, and to make this statement more precise we need to recall the definition of some notions of large deviations, for which the reference is the book of Dembo and Zeitouni (\cite{DZ}).
 We say that  a sequence $(P_n)$ of probability measures on a measurable Hausdorff  space $(\mathcal X, B(\mathcal X))$ satisfies the LDP at scale $u_n$ (with $u_n \rightarrow \infty$), if there exists a lower semicontinous function $I : {\mathcal X} \rightarrow [0, \infty]$ such that
\begin{eqnarray}
\label{LDPF}
\limsup \frac{1}{u_n} \log P_n (F) \leq -\inf\{I(x) ; x \in F\}\\
\label{LDPG}
\liminf \frac{1}{u_n} \log P_n (G) \geq -\inf\{I(x) ; x \in G\}\,,
\end{eqnarray}
for every closed set $F \subset {\mathcal X}$ and  every open set $G \subset {\mathcal X}$. The rate function $I$ is called good if its level sets are compact. More generally, a sequence of $\mathcal X$-valued random variables
is said to satisfy the
LDP if their distributions satisfy the LDP.
From now on, we work with ${\mathcal X} = {\mathcal M}_1(\mathbb T)$ whose compacity makes our task simpler.
In fact, according to \cite{DZ} Theorem 4.1.11, in this case the LDP is equivalent to the following property: for every $\mu \in {\mathcal M}_1(\mathbb T)$
\begin{eqnarray}\label{deux} -I(\mu) &=&\lim_{\varepsilon\downarrow 0}\limsup_n \frac{1}{u_n} \log P_n (B(\mu, \varepsilon))\\
&=& \lim_{\varepsilon\downarrow 0}\liminf_n \frac{1}{u_n} \log P_n (B(\mu, \varepsilon))\end{eqnarray}
where $B(\mu, \varepsilon)$ is the open ball of radius $\varepsilon$, centered at $\mu$
\[B(\mu, \varepsilon) = \{\nu : d_L(\mu, \nu) < \varepsilon\}\,.\]
It is easy to see that the latter is equivalent to the pair of inequalities
\ben\label{majalterb} \lim_{\varepsilon\downarrow 0}\limsup_n \frac{1}{u_n} \log P_n (B(\mu, \varepsilon)) &\leq& -I(\mu)\\
\label{minalterb}\lim_{\varepsilon\downarrow 0}\limsup_n \frac{1}{u_n} \log P_n (B(\mu, \varepsilon)) &\geq& -I(\mu)\,.
\een

 For the sake of completeness, let us explain shortly why (\ref{majalterb}) and  (\ref{minalterb}) lead to (\ref{LDPF}) and (\ref{LDPG}), respectively. From the one hand, every closed (hence compact) set in ${\mathcal M}_1 (\mathbb T)$ may be recovered by a finite number of balls
and applying
(\ref{majalterb}) for all these balls leads to (\ref{LDPF}). On the other hand, every open set contains open balls and it remains to
optimize with repect to the centers of the balls.

Our large deviations result follows the way initiated by the pioneer paper of Ben Arous and Guionnet (\cite{BenGui}) and continued by Hiai and Petz  (\cite{HiaiIHP},\cite{hiai2}).

We work with the set ${\mathcal M}_1(\mathbb T)$
 of probability measures on the unit circle.
  For 
$\mu \in   {\mathcal M}_1(\mathbb T)$, the Voiculescu entropy is
\ben\Sigma(\mu) = \int\!\!\int \log|\zeta- \zeta'| d\mu(\zeta)d\mu(\zeta')\,.\een
With a different sign it is the logarithmic energy of $\mu$ (see Chapter 5.3 in \cite{HiaiP}).
We also define  the potential
\ben
\nonumber
Q_\dd(\zeta) =
\begin{cases} - 2(\Re\!\ \dd) \log \big(2 \sin \frac{\theta}{2}\big) - (\Im\!\ \dd) (\theta- \pi)
 &\text{if $\zeta = e^{\ii\theta}, \theta \in (0, 2\pi)$}
\\
\infty & \text{if $\zeta =1$ and $\Re\!\ \dd > 0$}\\
- |\Im\!\ \dd| \pi & \text{if $\zeta =1$ and $\Re\!\ \dd = 0$}
\end{cases}
\\
\label{pot}
\een
It should be noticed that $Q_\dd$ is a lower semi continuous function (l.s.c.). The main result of this section is an extension of the theorem of Hiai and Petz (Theorem 5.4.10 in  \cite{HiaiP}), which corresponds to the case $\dd =0$ + a \textit{continuous} potential.

\begin{thm}\label{LDPESDalter}
Let $\dd \in \mathbb C$ with $\Re\!\ \dd \geq 0$.
For $n \in \mathbb N$ and $(\zeta_1, \cdots, \zeta_n) \in \mathbb T^n$
let
\[h(\zeta_1, \cdots, \zeta_n) :=  \prod_{k =1}^n (1- \zeta_k)^{\overline{\dd}\beta' n} (1 - \overline{\zeta_k})^{\dd\beta'n} \prod_{j<k} |\zeta_j - \zeta_k|^{2\beta'} \]
and let $\mathbb P_\dd\sn$ be the distribution on $\mathbb T^n$ having the density
\begin{equation}
\label{HPlambdaalter}
\frac{1}{{\mathcal Z}_\dd(n)} h(\zeta_1, \cdots, \zeta_n) 
\,,\end{equation}
where ${\mathcal Z}_\dd(n)$ is the normalization constant.
\begin{enumerate}
\item We have
\ben
\label{libre}
\lim_{n \rightarrow \infty} \frac{1}{n^2\beta'} \log {\mathcal Z}_\dd(n) = B(\dd)
\een
where
\be B(\dd) &=& \int_0^1 \left[x\log x + (x+ 2\Re\!\ \dd) \log (x+ 2\Re\!\ \dd)\right]\!\ dx\\
&& -\int_0^1\left[(x + \dd) \log(x+\dd) + (x +\bar\dd) \log (x+ \bar\dd)\right]\!\ dx\,.\ee
\item
The sequence of distributions of
\[\mu_\u\sn = \frac{\delta_{\zeta_1} + \dots +\delta_{\zeta_n}}{n}\]
under $\mathbb P_\dd\sn$ satisfies the LDP in ${\mathcal M}_1(\mathbb T)$
 at scale $\beta' n^2$ with (good) rate function 
\ben
\label{ratef}
I_\dd(\mu) = -\Sigma(\mu) +  \int_\mathbb T Q_\dd(\zeta) d\mu(\zeta)
+  B(\dd)\,.
\een
\item
The rate function vanishes only at $\mu= \mu_\dd$.
\end{enumerate}
\end{thm}

\medskip

\proof

(1) An exact expression of ${\mathcal Z}(n)$ is obtained using the following lemma, whose proof is postponed to the end of this subsection.

\begin{lem}\label{lem} If we define
\[{\mathcal Z}_{s,t}(n) = \int_{[0, 2\pi)^n} \prod_{k =1}^n (1- e^{\ii \theta_k})^s (1 - e^{- \ii \theta_k})^t \prod_{j <k} |e^{\ii \theta_j} - e^{\ii \theta_k}|^\beta d\theta_1 \dots d\theta_n,\]
then we have
\ben
\label{zst}
{\mathcal Z}_{s,t}(n)  = \frac{\Gamma(\beta' n +1)}{\big(\Gamma(\beta' +1)\big)^n} \prod_{0}^{n-1} \frac{\Gamma(\beta' j +1) \Gamma(\beta' j +1+s+t)}{\Gamma(\beta'j+1+s)\Gamma(\beta'j +1+t)}\,. \een
\end{lem}
\medskip

 We have ${\mathcal Z}_\dd(n) = {\mathcal Z}_{\overline{\dd} \beta' n, \dd \beta' n}(n)$
 and then, taking for $\log$ the principal value of the logarithm,
 \be
\log {\mathcal Z}_\dd(n) &=& \log\Gamma(\beta'n +1) -n \log\Gamma(\beta' +1)\\ 
&+&  \sum_{j=0}^{n-1}  \left[\log\Gamma(\beta'j +1)
+   \log\Gamma(\beta'j +1+2\Re\!\ \dd \beta'n)\right]\\ 
&-&  \sum_{j=0}^{n-1} \left[\log\Gamma (\beta'j +1+\dd \beta' n) + \log\Gamma (\beta'j +1+\bar \dd \beta' n)\right]\\
&+& 2\ii k_n \pi \ \ \ \ (k_n \in \mathbb Z  \ , \ |k| \leq 5n)
\,.
\ee
From the Binet formula (Abramowitz and  Stegun
\cite{astig} or Erd\'elyi et al. \cite{bateman} p.21), we have for $\Re\!\ x > 0$
\ben
\label{bin2}
\log\Gamma (x) = (x-\frac{1}{2})\log x -x + \frac{1}{2} \log(2\pi) + \int_0^\infty f(s)e^{-sx}\ \! ds\,.
\een
where the function $f$ is defined by
\be
 f(s) = \left[\frac{1}{2}-\frac{1}{s}+ \frac{1}{e^s -1}\right]\frac{1}{s} =
 2\sum_{k=1}^\infty \frac{1}{s^2 + 4\pi^2 k^2}\,,
\ee and satisfies for every $s \geq 0$ \be
 0 < f(s)
\leq f(0)= 1/12 \ , \  \ 0 < \left(sf(s)+ \frac{1}{2}\right) < 1\,.
\ee
Using  (\ref{bin2}), a straightforward study of Riemann sums gives
\ben
\nonumber
\lim_n \frac{1}{\beta'n^2} \log {\mathcal Z}_\dd(n) = 
 B(\dd)\,,
\een
where
\ben\nonumber B(\dd) &=& \int_0^1 \left[x\log x + (x+ 2\Re\!\ \dd) \log (x+ 2\Re\!\ \dd)\right]\!\ dx\\
\label{librebis}
&& -\int_0^1\left[(x + \dd) \log(x+\dd) + (x +\bar\dd) \log (x+ \bar\dd)\right]\!\ dx\,.\een

(2) The proof is based on the explicit form of the joint eigenvalue density.
We follow the lines of \cite{BenGui}, \cite{HiaiIHP}, \cite{hiai2} and \cite{HiaiP}.
 We skip the index $\u$ for notational convenience.
Our goal is the proof of the two inequalities (cf. (\ref{majalterb}) and (\ref{minalterb})), which hold for every $\mu \in {\mathcal M}_1(\mathbb T)$:
\begin{eqnarray}
\label{maj} \lim_{\varepsilon\rightarrow 0}\limsup_n \frac{1}{\beta'n^2} \log {\mathbb P}\sn_\dd (\mu\sn \in B(\mu, \varepsilon)) \leq -I_\dd(\mu)
\\
\label{min}
\lim_{\varepsilon\rightarrow 0}\liminf_n \frac{1}{\beta'n^2} \log {\mathbb P}\sn_\dd (\mu\sn \in B(\mu, \varepsilon)) \geq -I_\dd(\mu)
\end{eqnarray}
We rest on the LDP known in the case of $\dd =0$, and use the general method 
 which consists in estimating  the Radon-Nikodym derivative ruling the change of probability. The key formula is

\begin{equation}
\label{RN}{\mathbb P}\sn_\dd (\mu\sn \in B(\mu, \varepsilon))=
\frac{{\mathcal Z}_0(n)}{{\mathcal Z}_\dd(n)}\mathbb E_0\sn \left[1_{\mu\sn \in B(\mu, \varepsilon)} e^{-n^2\beta' \int Q_\dd d\mu\sn} \right]\,.\end{equation}
\textsl{The upper bound (\ref{maj})}.
Let us first assume $\Re\!\ \dd > 0$, so that $Q_\dd (\zeta) \rightarrow \infty$ as $\zeta \rightarrow 1$.
For $R > 0$ we consider the cutoff $Q^R = \min (Q_\dd, R)$. Since $Q^R$ is continuous, the mapping $\nu \in {\mathcal M}_1(\mathbb T) \mapsto \int Q^R\!\  d\nu$ is continuous, so
\ben\label{sci}
\inf_{B(\mu, \varepsilon)} \int Q^R d\nu \geq \int Q^R\!\ d\mu - r_1(\varepsilon, R)\een
with $\lim_\varepsilon r_1(\varepsilon, R) = 0$.
Since $Q_\dd \geq Q^R$ we get
\ben
\nonumber\frac{1}{\beta' n^2} \log \mathbb E_0\sn \left[1_{\mu\sn \in B(\mu, \varepsilon)} e^{-n^2\beta' \int Q_\dd d\mu\sn}\right]  &\leq& \frac{1}{\beta' n^2}\log \mathbb P_0\sn (\mu\sn \in B(\mu, \varepsilon))\\
 &-&\int Q^R\!\  d\mu +r_1(\varepsilon, R) \een
Thanks to the LDP known for $\dd = 0$, we can
take limsup in $n$ and then limit in $\varepsilon$ and obtain
\ben
\nonumber\lim_\varepsilon \limsup_n \frac{1}{\beta' n^2} \log \mathbb E_0\sn \left[1_{\mu\sn \in B(\mu, \varepsilon)} e^{-n^2\beta' \int Q_\dd d\mu\sn}\right]
\leq  \Sigma(\mu) - \int Q^R\!\  d\mu\,,
\een
and thanks to (\ref{RN}) and (\ref{libre})
\[\lim_\varepsilon \limsup_n \frac{1}{\beta' n^2} \log {\mathbb P}\sn_\dd (\mu\sn \in B(\mu, \varepsilon)) \leq \Sigma(\mu) - \int Q^R\!\  d\mu - B(\dd)\,.\]
By the monotone convergence theorem $\lim_{R\rightarrow \infty}  \int Q^R\!\  d\mu = \int Q_\dd d\mu$, which proves (\ref{maj}).

If $\Re\!\ \dd =0$, then $Q_\dd$ is lower semi continuous and bounded, so that the mapping $\nu \mapsto \int Q_\dd d\nu$ is lower semi continuous and (\ref{sci}) still holds with $Q_\dd$ instead of $Q^R$ and some $r(\varepsilon)$ instead of $r_1(\varepsilon, R)$. The rest of the argument is the same as above.

\textit{The lower bound (\ref{min})}.
If $I_\dd (\mu) = \infty$, the bound is trivial, so that
 we can assume that $\mu$ has no atom.
To overcome the problem of the singularity at $1$, we use a classical approximation of $\mu$ by a probability vanishing in a neighborhood of $1$, i.e.
\[d\mu_M(\zeta) = \frac{{\mathbf 1}_{|1-\zeta| \geq M^{-1}}}{\mu(|1-\zeta| \geq M^{-1})}  d\mu(\zeta)\,. \]
The benefit is that the mapping $\nu \mapsto \int Q_\dd\!\ d\nu$ is continuous in a neighborhood of $\mu_M$.
Choosing $M$ large enough, we ensure that $d_L(\mu_M , \mu) \leq \varepsilon/2$ and then, thanks to the triangle inequality
\[B(\mu, \varepsilon) \supset B(\mu_M, \varepsilon/2),\]
which leads to
\[{\mathbb P}\sn_\dd (\mu\sn \in B(\mu, \varepsilon)) \geq {\mathbb P}\sn_\dd (\mu\sn \in B(\mu_M, \varepsilon/2))\,.\]
We come back to (\ref{RN}) with $\mu_M$ and $\varepsilon/2$.
For  $\nu \in B(\mu_M, \varepsilon/2)$, we have $\int Q_\dd\!\ d\nu \leq \int Q_\dd\!\ d\mu_M + r_2(\varepsilon, M))$ where 
$\lim_\varepsilon r_2(\varepsilon, M) =0$. We get
\be
\nonumber\frac{1}{\beta' n^2} \log \mathbb E_0\sn \left[1_{\mu\sn \in B(\mu, \varepsilon)} e^{-n^2\beta' \int Q_\dd\!\ d\mu\sn}\right]&\geq&
\frac{1}{\beta' n^2} \mathbb P_0\sn (\mu\sn \in B(\mu_M, \varepsilon/2))\\ &-& \int Q_\dd\!\ d\mu_M - r_2(\varepsilon, M)\,.
\ee
Again, we take liminf in $n$ and lim in $\varepsilon$, and use  successively the LDP for $\dd =0$, (\ref{RN}) and (\ref{libre}) to obtain:
\[\lim_\varepsilon \limsup_n \frac{1}{\beta' n^2} \log {\mathbb P}\sn_\dd (\mu\sn \in B(\mu, \varepsilon)) \geq \Sigma(\mu_M) - \int Q_\dd\!\  d\mu_M - B(\dd)\,.\]
Now, since $I(\mu) < \infty$ and since $Q_\dd$ and $\Sigma$ are bounded below, the monotone convergence theorem yields
\[\lim_{M\rightarrow \infty} \Sigma(\mu_M) = \Sigma(\mu) \ \ , \ \ \lim_{M\rightarrow \infty} \int Q_\dd\!\ d\mu_M = \int Q_\dd\!\ d\mu\,.\]
 This ends the proof of (\ref{min}).


(3) The uniqueness of the minimizer is a direct consequence of the strict convexity of $I$ which comes from the strict concavity of $\Sigma$.

We do not give a self contained proof of the identity of the minimizer, but rather use a probabilistic argument. On the one hand in Theorem \ref{cvesdalter} we proved that $\mu_\u\sn$ converges weakly in probability to $\mu_\dd$,  and on the other hand the LDP combined with the uniqueness of the minimizer imply that $\mu_\u\sn$ converges weakly in probability to this minimizer. This completes the proof. $\Box$
\medskip

\begin{proof}[Proof of Lemma \ref{lem}]
We have
\[\frac{{\mathcal Z}_{s,t}(n)}{{\mathcal Z}_{0,0}(n)}  = \mathbb E \big(\det (\Id-U)^s \det(\Id - \bar U)^t \big)\]
where the mean is taken under the  $\hbox{CJ}\sn_{0, \beta}$
 distribution. 
  Under this distribution, $\det (\Id-U)$ has the same law as  the product of independent variables $1-\bar\alpha_k$, where  $\alpha_k$ is $\nu_{\beta(n-k-1)+1}$ distributed. We get
\[\mathbb E \big(\det (\Id-U)^s \det(\Id -\bar U)^t \big) = \prod_{j=0}^{n-1} \mathbb E (1- \bar\alpha_j)^s (1-\alpha_j))^t\]
From (\ref{laphua})
we get
\[\frac{{\mathcal Z}_{s,t}(n)}{{\mathcal Z}_{0,0}(n)}  = \prod_{0}^{n-1} \frac{\Gamma(\beta' j +1) \Gamma(\beta' j +1+s+t)}{\Gamma(\beta'j+1+s)\Gamma(\beta'j +1+t)}\]
Besides, Lemma 4.4 in \cite{KillipNenciu} gives
\[{\mathcal Z}_{0,0}(n) = \frac{\Gamma(\beta' n +1)}{\big(\Gamma(\beta' +1)\big)^n}\,.\qedhere\]
\end{proof}
\renewcommand{\refname}{References}

\end{document}